\documentclass[final]{siamart1116}

\title{A-posteriori error estimation and adaptivity for nonlinear parabolic equations using IMEX-Galerkin discretization of primal and dual equations
\thanks{Submitted to SIAM Journal on Scientific Computing\funding{This work was supported by the Netherlands Organisation for Scientific Research (NWO) via the Innovational Research Incentives Scheme (IRIS), Veni grant 639.031.033.}}}

\author{
  X.~Wu\thanks{Multiscale Engineering Fluid Dynamics, Eindhoven University of Technology, P.~O.~Box 513, 5600 MB Eindhoven, Netherlands
    (\email{x.wu@tue.nl}, \email{g.simsek@tue.nl}, \email{e.h.v.brummelen@tue.nl}).}
  \and
  K.~G.~van der Zee\thanks{ School of Mathematical Sciences, University of Nottingham, University Park, NG7 2RD Nottingham, United Kingdom (\email{kg.vanderzee@nottingham.ac.uk).}}
  \and
  G.~Simsek\footnotemark[2]
    \and
  E.~H.~van Brummelen\footnotemark[2]
}

\newlength{\bigfboxsep}
\newcommand{\norm}[1]{{\|#1\|}}

\newcommand{\dual}[1]{\langle#1\rangle}

\newcommand{\ds}[1]{{\displaystyle{#1}}}

\newcommand{\pd}{\partial}
\newcommand{\dd}{\mathrm{d}}

\newcommand{\mbf}[1]{\mathbf{#1}}

\newcommand{\mbfN}{\mbf{N}}

\newcommand{\mcal}[1]{\mathcal{#1}}

\newcommand{\mcB}{\mcal{B}}

\newcommand{\mcE}{\mcal{E}}

\newcommand{\mcI}{\mcal{I}}

\newcommand{\mcK}{\mcal{K}}

\newcommand{\mcN}{\mcal{N}}

\newcommand{\mcQ}{\mcal{Q}}
\newcommand{\mcR}{\mcal{R}}
\newcommand{\mcS}{\mcal{S}}

\newcommand{\mcV}{\mcal{V}}
\newcommand{\mcW}{\mcal{W}}

\usepackage{lipsum}
\usepackage{amsfonts}
\usepackage{graphicx}
\usepackage{color}
\usepackage{subfig}
\usepackage{multirow}
\usepackage{algpseudocode}
\usepackage{algorithm,caption}
\usepackage{amsmath}
\usepackage{amsopn}

\algnewcommand{\algorithmicgoto}{\textbf{go to}}%
\algnewcommand{\Goto}[1]{\algorithmicgoto~\ref{#1}}%

\newcommand{\clippedFig}[1]{\includegraphics[scale=0.22,viewport=100 40 480 400,clip]{#1}}

\newcommand{\clippedPlot}[1]{\includegraphics[scale=0.142]{#1}}

\newcommand{\Plott}[1]{\includegraphics[scale=0.234]{#1}}
\newsiamremark{remark}{Remark}
\crefname{remark}{Remark}{Hypotheses}
\numberwithin{theorem}{section}
\ifpdf
  \DeclareGraphicsExtensions{.eps,.pdf,.png,.jpg}
\else
  \DeclareGraphicsExtensions{.eps}
\fi

\begin{document}

\maketitle

\begin{abstract}
While many methods exist to discretize nonlinear time-dependent partial differential equations (PDEs), the rigorous estimation and adaptive control of their discretization errors remains challenging. In this paper, we present a methodology for duality-based a posteriori error estimation for nonlinear parabolic PDEs, where the full discretization of the PDE relies on the use of an implicit-explicit (IMEX) time-stepping scheme and the finite element method in space. The main result in our work is a decomposition of the error estimate that allows to separate the effects of spatial and temporal discretization error, and which can be used to drive adaptive mesh refinement and adaptive time-step selection. The decomposition hinges on a specially-tailored IMEX discretization of the dual problem. The performance of the error estimates and the proposed adaptive algorithm is demonstrated on two canonical applications: the elementary heat equation and the nonlinear Allen--Cahn phase-field model.
\end{abstract}
%


\setlength{\belowcaptionskip}{-3pt}

\section{Introduction}
Nonlinear parabolic PDEs are ubiquitous in science, however, their efficient numerical solution remains challenging. Implicit-explicit (IMEX) methods have been widely used for the time integration of complex time-dependent PDEs with terms of different type \cite{ascher2008numerical, boscarino2016high}. Recently, a number of IMEX time-stepping schemes, paired with spatial Galerkin finite-element discretizations, have been proposed for phase-field models~\cite{WisWanLowSINUM2009, GomHugJCP2011, WuZwiZeeIJNMBE2014, TieGuiACME2015, VigDalBroColCalCS2015}, which are currently a much-studied class of nonlinear parabolic problems~\cite{KimLowIFB2005, OdeHawPruM3AS2010, HilKamNguZeeM3AS2015, SheYanSINUM2015, GomZeeBOOK-CH2016}.
When the PDE solution displays alternating fast and slow variations, the numerical discretization can, obviously, benefit significantly from adaptivity in both space and time. 
\par
This paper is devoted to the development of a posteriori error estimates and corresponding adaptive algorithms for these popular discretizations. In particular, we consider dual-based error estimates that assess the discretization error with respect to user specified quantities of interest describing the goal of the analyses. The quantities of interest might, for instance, be physical quantities or some appropriate norms of the error of the solution (e.g. energy norm, $L^2$ norm). To efficiently drive adaptive mesh refinement and adaptive time-step selection, the error estimates need to address the temporal and the spatial discretization errors separately.
\par
There have been several studies on goal-oriented adaptive techniques for parabolic equations during the last decade, but mostly in the context of space-time (discontinuous) Galerkin finite element discretization, see for instance Eriksson and Johnson \cite{EriJohSINUM1991, EriJohSINUM1995a, EriJohSINUM1995b, EriJohSINUM1995c}, Schmich and Vexler \cite{SchVexSISC2008}, Carey et al. \cite{CarEstJohLarTavSISC2010}, Bermejo and Carpio \cite{BerCarSISC2009}, Braack et al. \cite{BraBurTasSISC2011}, Besier and Rannacher\cite{besier2012goal}, and Asner et al. \cite{asner2012adjoint}. Very little progress has been made for parabolic equations discretized using IMEX time-stepping schemes. 
\par
Recently, Chaudhry et al. \cite{chaudhry2015error, chaudhry2015posteriori} proposed a posteriori error estimates for various IMEX schemes, based on an equivalence relation between IMEX schemes and \emph{time-Galerkin} finite element methods. They rewrite the time-Galerkin method using special numerical quadrature rules and carry out a standard duality-based analysis for the resultant approximations. The splitting of the temporal and the spatial error contributions in these error estimates are commonly achieved by inserting and subtracting suitable projections of the dual solution.
\par
The objective of this paper is to present an alternative approach to duality-based a posteriori error estimates for fully discretized semi-linear parabolic PDEs using conforming finite elements in space and first-order IMEX schemes in time. Contrary to Chaudhry et al, in our approach we \emph{directly} obtain a posteriori error estimates without resorting to an interpretation of IMEX as a Galerkin-in-time method. This paper is a follow-up to our recent paper~\cite{csimcsek2015duality}, where we only considered errors due to spatial discretization. The focus of this work is on the \emph{total} discretization error which contains both the spatial and temporal parts. 
\par
The starting point of our analysis is the exact duality-based error representation, which is a duality pairing of the global space-time residual with the solution of the mean-value-linearized (backward-in-time) dual problem. This error representation can be decomposed into various distinct residuals weighted by the same dual solution. A fundamental framework for successfully decomposing the residuals for (non)linear parabolic PDEs, discretized by a classical A-stable $\theta$-scheme in time, has been developed by Verf{\"{u}}rth \cite{verfurth2013posteriori} in the context of \emph{energy}-based a posteriori error analysis. 
\par
By extending Verf{\"{u}}rth's framework to IMEX schemes and a \emph{duality}-based error analysis, we will decompose the error representation into three contributions which can be associated to the temporal and spatial discretization error, and additionally data oscillation. This novel decomposition hinges on a special nonstandard IMEX discretization of the dual problem. We then propose a general space-time adaptive algorithm for an efficient distribution of the discretization parameters: a set of time steps and the refined mesh at each time step. 
\par
This work is structured as follows. In \Cref{sec:abstract}, we introduce the abstract setting for a general (non)linear parabolic PDE and its IMEX-Galerkin discretization. \Cref{sec:section_dep} is devoted to the methodology for a space-time decomposition of a duality-based a posteriori error estimate. After having established computable error estimates in \Cref{sec:computable}, we propose the associated adaptive algorithm in \Cref{sec:adp_alg}. The application to the elementary heat equation and the nonlinear Allen--Cahn equation (an elementary phase-field model), together with numerical results are presented in \Cref{sec:applications}, after which we present our conclusions.

\section{Abstract setting} \label{sec:abstract}
In this section, we start by introducing an abstract setting of nonlinear parabolic PDEs and the corresponding dual problem in weak formulations. Then we present the discretization of the primal problem using IMEX time-stepping schemes and conforming finite elements in space.
\subsection{Weak formulation and error representation} 
As a model problem we consider a general semi-linear parabolic equation in a bounded space-time domain $\Omega \times (0,T]$ with $\Omega \subset \mathbb{R}^d$, $d = 1, 2, 3$, having natural boundary conditions. To provide a setting for the weak formulation, we denote by $\mathcal{V}$ a suitable Hilbert space and by $\mathcal{V}^{*}$ its dual space, such that $\mcV \subset L^2(\Omega) \subset \mcV^{*}$ with continuous embeddings. We denote the inner product in $L^2(\Omega)$ by $(\cdot,\cdot)$, and duality pairings between $\mcV^{*}$ and $\mcV$ by $\left\langle \cdot , \cdot \right\rangle$.
By defining $\mathcal{W} := \left\{  v \in L^2(0,T; \mathcal{V}), \partial_t v \in L^2(0,T; \mathcal{V}^{*}) \right\}$ as a suitable space for $u$, the weak form reads: find $u \in \mathcal{W}_{u^0} := \left\{  v \in \mathcal{W} : v(0) = u^0 \right\}$ such that $\forall v \in L^2(0,T; \mathcal{V}) $ 
\begin{equation}\label{eq:p_wk}
\int_0^T \Big( \langle \partial_{t} u, v \rangle + \mathcal{B}(u,v) + \mathcal{N}(u;v) \Big) \, \dd t= \int_0^T \langle f, v \rangle \, \dd t  
\end{equation}
where $f \in L^2(0,T; \mathcal{V}^{*})$, $u^0 \in L^2(\Omega)$, the semi-linear form $\mathcal{N}(\cdot ;\cdot)$ of a sufficiently smooth nonlinear operator represents the nonlinear components which is linear with respect to arguments on the right of the semicolon, and $\mathcal{B}(\cdot ,\cdot)$ is the bilinear form of a elliptic self-adjoint operator. A prime example of the abstract setting is the Allen--Cahn equation $\partial_t u - \Delta u + \frac{1}{\varepsilon^2}\psi'(u)=0$, which will be discussed later in \Cref{sec:allen}.
\par
Given the solution $u$, we consider the quantity
\begin{equation}\label{eq:qoi}
\mcQ(u) 
:= \left(\bar{q}, u(T)\right)
+ \int_0^T (q, u) \, \dd t,
\end{equation}
with $\bar{q} \in L^2(\Omega)$ and $q\in L^2(0,T;L^2(\Omega))$ so that $\mcQ: \mathcal{W} \to \mathbb{R}$ is a continuous linear functional.\footnote{One can more generally consider $q\in L^2(\tau_1,T;\mcV^*) \cup L^2(0,T;L^2(\Omega))$ where $\tau_1>0$ is the size of the first time step. For technical reasons later on (i.e., $\dual{q_0,u_0}$ must be well-defined, with $q_0$ defined in \cref{eq:q_0}), we can not take $\tau_1=0$.} One example of $\mcQ(u)$ would be the value of the solution at the final time $t=T$ at a critical area of the domain centered at $\mathbf{x}_0$, 
$
\mcQ(u) 
= \int_\Omega \rho_\epsilon (\mathbf{x}_0 - \mathbf{x}) u(\mathbf{x}, T) \, \dd \mathbf{x},
$
where $\rho_\epsilon \in C^\infty$ is a kernel function with radius and center of $\epsilon$ and $\mathbf{x}_0$. Alternatively, one might wish to estimate the error in the $L^2$ norm at the final time $T$. To achieve this, we set $\bar{q} = u(T) - \hat{u}(T)$ and $q=0$ where $\hat{u}$ is an approximation of the solution $u$. Then we have $\mcQ(u) - \mcQ(\hat{u}) = \norm{u(T)-\hat{u}(T)}^2_{L^2(\Omega)}$.
\par
For any $u, \hat{u} \in \mcV$, we denote by $\mathcal{N}^s(u,\hat{u};\cdot,\cdot)$ the mean-value linearization of $\mathcal{N}(\cdot,\cdot)$ performed at a value in between $u$ and $\hat{u}$, namely,
\begin{align}\label{eq:no}
\mathcal{N}^s(u,\hat{u};w,v) = \int_0^1 \mathcal{N}'\big(su + (1-s)\hat{u}\big) (w, v) \dd s ,  \qquad \forall w,v \in \mathcal{V}
\end{align}
where $\mathcal{N}'$ is the G\^{a}teaux derivative of $\mathcal{N}$, i.e.
$$
 \mathcal{N}'(\hat{w})(w,v) = \lim_{s \to 0} \frac{\mathcal{N}\left(\hat{w}+ s \, w ; v\right) - \mathcal{N}\left(\hat{w};v\right)}{s} \dd s, \qquad \forall w, v \in \mathcal{V}.
$$
Note that if we set $w=u-\hat{u}$, the chain rule gives
\begin{align}\label{eq:mean_l} 
\mathcal{N}^s(u,\hat{u};u-\hat{u},v)  = \mathcal{N}(u;v) - \mathcal{N}(\hat{u};v).  
\end{align}
The mean-value-linearized (backward-in-time) dual problem takes the form: find $z \in \mathcal{W}^{\bar{q}}:= \left\{  v \in \mathcal{W} : v(T) = \bar{q} \right\}$ such that $\forall w \in L^2(0,T; \mathcal{V})$
\begin{equation}\label{eq:d_wk}
\int_0^T \Big(  \langle - \partial_{t} z, w  \rangle + \mathcal{B}(z, w) + \mathcal{N}^s(u,\hat{u};w,z) \Big) \, \dd t=  \int_0^T (q,w) \, \dd t
\end{equation}
Let $\hat{u} \in \mcW$ denote any approximation of the solution $u$ in \cref{eq:p_wk}. We define the residual of the primal PDE, $\mathcal{R}_\mathrm{PDE}$, and the residual of the initial condition, $\mathcal{R}_0$, as
\begin{align} \label{eq:R_st}
\mathcal{R}_\mathrm{PDE}  (\hat{u}(t);v) &:= \langle f(t),v \rangle - \langle \partial_{t} \hat{u}(t), v \rangle - \mathcal{B}(\hat{u}(t),v) - \mathcal{N}(\hat{u}(t);v) \\
\mathcal{R}_0(\hat{u}(0);w) &:= \left( u^0 - \hat{u}(0), w\right) \label{eq:R_0}
\end{align}
for all $v \in \mcV$ and $w \in L^2(\Omega)$. Following the general framework of goal-oriented error analysis (see, e.g. \cite{BecRanAN2001, OdePruCMA2001}), we obtain an exact error representation assessing the error in $\mcQ$, which can generally be represented by a global space-time residual weighted by the solution of the dual problem~\cref{eq:d_wk}.
\begin{theorem}[Global space-time error representation]\label{thm:error_theorem}
Given any approximation $\hat{u} \in \mcW$ of the solution $u$ of the primal problem \cref{eq:p_wk}, we have the following a posteriori error representation:
\begin{equation}\label{eq:q_st_r}
\mcQ(u) - \mcQ(\hat{u} ) = \mathcal{R}_0\big( \hat{u}(0); z(0)\big) + \int_0^T  \mathcal{R}_\mathrm{PDE}  \left( \hat{u}(t); z(t)\right) \, \dd t 
\end{equation}   
where $z \in \mcW^{\bar{q}}$ is the solution of the dual problem \cref{eq:d_wk}.
\end{theorem}
\begin{proof}
The proof is standard, see, e.g. \cite{van2011goal} or \cite[Theorem 2.A]{csimcsek2015duality}. 
\end{proof}
Note that errors in norm are also included in \cref{thm:error_theorem} by suitably changing $\bar{q}$ and $q$, e.g., as in the example above.
\subsection{IMEX~-~FEM Discretization}\label{sec:section_dis}
We next describe a full discretization of problem \cref{eq:p_wk} by partitioning $[0,T]$ as $0=t_0<t_1<t_2<\cdots<t_k<\cdots<t_N=T$ into $N$ subintervals $\mathcal{I}_{k+1}=[t_k,t_{k+1}]$ of length $\tau_{k+1} = t_{k+1} - t_k$, $k=0,1,\cdots,N-1$. Because of the nonlinearity in the system \cref{eq:p_wk}, one has to be careful in choosing a time discretization to avoid prohibitive stability restrictions and high computational complexity. In this paper,  we focus on first-order IMEX time-stepping schemes, which employ a splitting of the nonlinear term $\mcN$ according to 
$$
\mathcal{N}(u; v) = \mathcal{N}_c(u; v) - \mathcal{N}_e(u; v) ,
$$
The notation $\mathcal{N}_c$ and $\mathcal{N}_e$ comes from the phase-field modeling community, and refers to the contractive and expansive part, respectively, which can also refer to the stiff and non-stiff term. The fundamental idea is to treat the contractive part \emph{implicitly} and the expansive part \emph{explicitly}. Such a time scheme for problem \cref{eq:p_wk} is defined recursively by: find $u_{k+1} \in \mathcal{V}$ such that $\forall v\in \mathcal{V} $
\begin{equation}\label{eq:p_time}
\left( \frac{u_{k+1} - u_k}{\tau_{k+1}} , v\right) + \mathcal{B}(u_{k+1}, v) + \mathcal{N}_c(u_{k+1}; v) - \mathcal{N}_e(u_{k}; v) = (f_{k+1}, v )
\end{equation}
for $k = 0,1,\cdots , N-1$, where the initial condition is 
\begin{align}\label{eq:u0}
(u_0,v) = (u^0, v) \qquad \forall v \in L^2(\Omega).
\end{align}
Here, $f_{k+1} = f(\cdot,t_{k+1})$, which is well-defined upon assuming that the function $f$ is sufficiently regular, e.g., $f\in C^0((0,T];L^2(\Omega))$. We remark that instead of the time approximation $f_{k+1}$, a time-averaged approximation $\bar{f} = \frac{1}{\tau_{k+1}}\int_{t_{k}}^{t_{k+1}} f(\cdot, t) \, \dd t$ can be used, provided that $f\in L^2(0,T;L^2(\Omega))$. We also implicitly assume in \cref{eq:p_time} that $\mcN_e(u_0;v)$ is bounded for $u_0 \in L^2(\Omega)$. If $\mcN_e(u_0;v)$ is not well-defined, one can remove this term from \cref{eq:p_time} for the first time step. For simplicity, we continue our analysis assuming that $\mcN_e(\cdot;\cdot)$ is bounded on $L^2(\Omega)\times L^2(\Omega)$.    
\par
To fully discretize the primal problem \cref{eq:p_wk}, we consider a standard shape-regular mesh $\mcK_{k}$ of $\Omega$ and an associated conforming finite element space $\mathcal{S}^{h,p}_{k}$ defined by 
$$
\mcS^{h,p}_{k} := \left\{  v \in \mcV: v(x) \vert_{K} \in \mathbb{P}^p (K), \forall K \in \mcK_k \right\}
$$
for $k=0,1,\ldots,N$, where $\mathbb{P}^p (K)$ is the space of polynomials up to order $p$ on element $K$ and $h$ denotes the mesh parameter. The fully discrete approximation is then formulated as: find $u_{k+1}^h \in \mathcal{S}^{h,p}_{k+1}$ such that $\forall v^h \in \mathcal{S}^{h,p}_{k+1} $ 
\begin{equation}\label{eq:p_full}
\left( \frac{u_{k+1}^h - u_k^h}{\tau_{k+1}} , v^h \right) + \mathcal{B}(u_{k+1}^h, v^h) + \mathcal{N}_c(u_{k+1}^h; v^h) - \mathcal{N}_e(u_{k}^h; v^h) = (f_{k+1}, v^h )
\end{equation}
for $k = 0,1,\cdots , N-1$, where the initial condition is 
\begin{equation}\label{eq:u_0h}
(u_0^h,v^h) = (u^0, v^h) \qquad \forall v^h \in \mathcal{S}^{h,p}_0.
\end{equation}
\par
We assume that the solutions $u_{\tau}:=\{u_k\}_{k=0}^N$ and $u_{\tau h}:=\{u_k^h\}_{k=0}^N$ exist for the time-discrete primal problem~\cref{eq:p_time} and the fully-discrete primal problem \cref{eq:p_full}, respectively. 

\section{Space-time decomposed a posteriori error estimate}\label{sec:section_dep}
Space-time adaptivity is heavily dependent on an appropriate decomposition of error estimates, which will be derived in this section. Our approach to isolate error contributions from different sources is inspired by the work of Verf\"{u}rth in \cite[Chapter 6]{verfurth2013posteriori}, which contains a general framework for deriving residual-based a posteriori error estimates for nonlinear parabolic problems with the $\theta$-scheme. In the following Lemma, we adapt Verf\"{u}rth's residual decomposition to our fully discrete primal problem \cref{eq:p_full}.\par
\begin{lemma}[Residual decomposition]\label{lem:residual_decomp}
Let $u_{\tau h} := \lbrace u_k^h \rbrace_{k=0}^N $ denote the solution of the fully discrete problem \cref{eq:p_full}, and $I u_{\tau h}$ denote the piecewise-linear time reconstruction of $u_{\tau h}$ on time intervals $[t_{k},t_{k+1}]$, $k=0,1,\ldots,N-1$, i.e., 
\begin{equation}
I u_{\tau h} (t) = \frac{t_{k+1} - t}{\tau_{k+1}} u_k^h + \frac{t-t_k}{\tau_{k+1}}u_{k+1}^h \qquad t \in [ t_k, t_{k+1}].
\end{equation}
Let the spatial residual $r_h^k$, the temporal residual $r_{\tau}^k $ and the data-oscillation contribution $r_f^k$ be defined, for each $k = 0,1,\cdots , N-1$, by
\begin{align}
\left\langle r_h^{k+1}, v \right\rangle :=& ( f_{k+1}, v) - \left( \frac{u^h_{k+1} - u^h_k}{\tau_{k+1}}, v\right)  -\mathcal{B}(u^h_{k+1}, v) - \mathcal{N}_{c}(u^h_{k+1}; v) + \mathcal{N}_{e} (u^h_{k};  v) \label{eq:R_s}\\
\left\langle r_{\tau}^{k+1}(t), v \right\rangle :=&\mathcal{B}(u^h_{k+1}, v) + \mathcal{N}_{c}(u^h_{k+1}; v) - \mathcal{N}_{e} (u^h_{k};  v)-\mathcal{B}(I u_{\tau h}(t),v) - \mathcal{N}(I u_{\tau h}(t);v)\label{eq:r_t}\\
\left\langle r_f^{k+1}(t), v \right\rangle :=& \left\langle f(t) - f_{k+1}, v \right\rangle. \label{eq:r_f}
\end{align}
for all $v \in \mcV$ and $t\in (t_k,t_{k+1}]$.
Then, for each $k = 0,1,\cdots , N-1$, the following decomposition of the space-time residual \cref{eq:R_st} holds:
\begin{align}
\mcR_\mathrm{PDE}(I u_{\tau h}(t);v) &=   \langle f(t),v \rangle - \langle \partial_{t} I u_{\tau h}(t), v \rangle - \mathcal{B}(I u_{\tau h}(t),v) - \mathcal{N}(I u_{\tau h}(t);v) \label{eq:R_st_uI}\\
&= \left\langle r_h^{k+1} , v\right\rangle  + \left\langle r_{\tau}^{k+1}(t), v \right\rangle + \left\langle r_f^{k+1}(t), v \right\rangle \label{eq:r_dep}
\end{align}
where $t \in (t_k, t_{k+1}]$.
\end{lemma}
\begin{proof}
Since $\partial_t I u_{\tau h} =  \frac{u^h_{k+1} - u^h_k}{\tau_{k+1}}$ on $(t_k, t_{k+1}]$, the identities in \cref{eq:r_dep} follow from a straightforward substitution in \cref{eq:R_st_uI} using the definition \cref{eq:R_s}, \cref{eq:r_t} and \cref{eq:r_f}.
\end{proof}
\begin{remark}\label{remark_residual}
We note that the spatial residuals \cref{eq:R_s} are independent of time, and due to Galerkin orthogonality, the spatial residuals will be equal to zero if $v \in \mcS^{h,p}_{k+1}$. Furthermore, upon convergence $u_k^h \to u_k$ as $h \to 0^+$, for all $k$, we also have $r_h^{k+1} \to 0$ (see \cref{eq:p_time}). Similarly, assuming sufficient smoothness in time, then $u_k^h, u_{k+1}^h \to I u_{\tau h} (t)$ for $t \in [t_k, t_{k+1}]$ as $\tau_{k+1} \to 0^+$, which implies $r_{\tau}^{k+1} (t) \to 0$ as $\tau_{k+1} \to 0^+$. This is the motivation for calling $r_{\tau}^{k+1}$ and $r_h^{k+1}$ the \emph{temporal residual} and the \emph{spatial residual}, respectively. 
\end{remark}
\subsection{Time-discrete error representation}
The first step toward a decomposition of duality-based error estimates is to introduce a time-discrete error representation identifying only the spatial discretization error. To this end, we introduce a novel and specially-tailored IMEX time-discrete dual problem. This time-discrete problem is driven by the following discrete representation of $\mcQ$. \par
Let us rewrite the piecewise-linear time reconstruction $I w_{\tau} \in \mcW$ of any sequence $w_{\tau}:=\{w_k\}_{k=0}^N$, $w_k \in \mcV$, as
\begin{equation}\label{eq:u_interp}
I w_{\tau}(\mathbf{x},t) = \sum_{k=0}^{N} w_k(\mathbf{x}) \mbfN_k(t)
\end{equation}
where
\begin{align*}
\mbfN_k(t)  :=  \left\lbrace \begin{array}{cl}
                           \ds{ \frac{t_{k+1} - t}{\tau_{k+1}} } & \text{if} \; t \in \mcI_{k+1}, k \leq N-1\\[8pt]
                          \ds{  \frac{t-t_{k-1}}{\tau_{k}} } & \text{if} \; t \in \mcI_{k}, k\geq 1 \\
                            0 & \text{otherwise}\\
                         \end{array} \right.
\end{align*}
We consider the following discrete representation of $\mcQ: \mcW \to \mathbb{R}$ when applied to $I w_{\tau}$.
\begin{lemma}\label{lem:lemma_q}
Let us define 
\begin{alignat}{2}\label{eq:q_k}
 q_{k} = \frac{1}{\tau_k} \int_0^T q \mbfN_{k} (t) \dd t \qquad \text{for } k=1,2,\ldots,N 
\end{alignat}
and 
\begin{alignat}{2}\label{eq:q_0}
q_0 = \frac{1}{\tau_1} \int_0^T q \mbfN_0(t) \dd t. 
\end{alignat}
Then, the following time-discrete representation of $\mcQ: \mcW \to \mathbb{R}$ holds
\begin{alignat}{2}\label{eq:qoi_discrete}
  \mcQ(I w_{\tau}) = \tau_1 (q_0,w_0)+ \sum_{k=1}^{N} \tau_k (q_{k},w_{k}) + \left(\bar{q}, w_N\right).
\end{alignat}
\end{lemma}
\begin{proof}
For $I w_{\tau}$ defined in \cref{eq:u_interp}, we observe that, according to \cref{eq:qoi}, 
\begin{equation*}
\mcQ(I w_{\tau}) =\left(\bar{q}, w_N\right) + \int_0^T \left( q,\; \sum_{k=0}^{N} w_k \mbfN_k(t) \right) \dd t = \left(\bar{q}, w_N\right)  +  \sum_{k=0}^{N} \left( \int_0^T q \, \mbfN_k(t) \, \dd t, \; w_k \right).
\end{equation*}
By virtue of
$$
\sum_{k=0}^{N} \left( \int_0^T q \, \mbfN_k(t) \, \dd t, \; w_k \right) = \left( \int_0^T q \mbfN_0(t) ,w_0 \right)+ \sum_{k=1}^{N} \left( \int_0^T q \, \mbfN_{k}(t) \, \dd t, \; w_{k} \right),
$$
we obtain \cref{eq:qoi_discrete} by substituting the definition \cref{eq:q_k} and \cref{eq:q_0}.
\end{proof}
\par
We now state the novel IMEX time-stepping scheme to discretize the dual problem backwards in time: Find $z_k \in \mcV$, $k=0,1,\ldots,N,$ such that 
\begin{equation}\label{eq:dual_time0}
-\left( \frac{z_1-z_{0}}{\tau_1}, w \right) -\mathcal{N}^s_{e}(u_{0},u_{0}^h; w, z_{1})= (q_0, w) \qquad \forall w \in \mathcal{V}
\end{equation}
and for $k = 1,2, \ldots, N-1$:
\begin{multline} \label{eq:dual_time1}
-\left( \frac{z_{k+1} - z_k}{\tau_{k}}, w\right) +  \mathcal{ B}(z_k, w) +  \mathcal{N}^s_{c}(u_{k},u_{k}^h; w, z_k) \\
- \frac{\tau_{k+1}}{\tau_k}\mathcal{N}^s_{e}(u_{k},u_{k}^h; w, z_{k+1}) = ( q_{k}, w ) \qquad \forall w \in \mcV
\end{multline}
 where the terminal condition is 
\begin{equation}\label{eq:dual_time2}
\left( z_{N}, w\right) +  \tau_{N}\mathcal{ B}(z_{N}, w) + \tau_{N}\mathcal{N}^s_{c}(u_{N},u_{N}^h; w, z_{N})  = \tau_{N}( q_{N}, w ) + (\bar{q}, w)\qquad \forall w \in \mcV.
\end{equation}
The time-discrete dual \cref{eq:dual_time0}-\cref{eq:dual_time2} has been defined so as to provide an exact error representation for $\mcQ(I u_{\tau h})$ with respect to $\mcQ(I u_{\tau})$. 
\begin{theorem}[Time-discrete error representation]\label{thm:discrete_theorem}
Let $u_{\tau} = \lbrace u_k \rbrace_{k=0}^N$ denote the solution of the time discrete system \cref{eq:p_time}, and $u_{\tau h} = \lbrace u_k^h \rbrace_{k=0}^N$ denote the solution of the fully discrete system \cref{eq:p_full}. Let $z_{\tau} = \lbrace z_k \rbrace_{k=0}^N$ denote the time discrete approximation of the dual problem obtained from \cref{eq:dual_time0}-\cref{eq:dual_time2}. 
Then the following error representation holds:
\begin{equation}\label{eq:qoi_s}
\mcQ(I u_{\tau}) - \mcQ(I u_{\tau h})  =\left(u^0-u_0^h, z_{0}-v_0^h\right) + \sum_{k=1}^{N} \tau_k \left\langle r_h^k , z_k-v_k^h\right\rangle,
\end{equation}
for any $v_k^h \in \mcS^{h,p}_{k}$, $k=0,1,\ldots, N$.
\end{theorem}
\begin{proof}
From \cref{eq:qoi_discrete}, it follows that $\mcQ(I u_\tau) - \mcQ(I u_{\tau h}) $ can be formulated as 
\begin{align}
\mcQ(I u_\tau) - \mcQ(I u_{\tau h}) = \tau_1( q_0,u_0-u_0^h)  + (\bar{q}, u_N - u_N^h)  + \sum_{k=1}^{N} \tau_k( q_{k}, u_{k} - u_{k}^h ) \label{eq:q_h}
\end{align}
Substituting the time-discrete dual problem \cref{eq:dual_time0}--\cref{eq:dual_time2} into \cref{eq:q_h}, we get
\begin{align*}
&\mcQ(I u_\tau) - \mcQ(I u_{\tau h})\\
&\quad= \tau_1 \Bigg\{ -\left( \frac{z_1-z_{0}}{\tau_1},u_0-u_0^h \right) - \mcN_e^s(u_0,u_0^h;u_0-u_0^h,z_1)\Bigg\}  \\
&\qquad + \tau_N  \Bigg\lbrace \bigg(\frac{z_{N} }{\tau_{N}}, u_{N} - u^h_{N} \bigg) + \mathcal{B}(z_{N}, u_{N} - u^h_{N}) + \mathcal{N}^s_{c}(u_{N},u_{N}^h; u_{N} - u^h_{N}, z_{N}) \Bigg\rbrace\\
&\qquad + \sum_{k=1}^{N-1} \tau_k  \Bigg\lbrace - \bigg(\frac{z_{k+1} - z_k}{\tau_k}, u_{k} - u^h_{k} \bigg)  + \mathcal{B}(z_k, u_{k} - u^h_{k})\\
&\qquad  + \mathcal{N}^s_{c}(u_{k},u_{k}^h; u_{k} - u^h_{k}, z_k) - \frac{\tau_{k+1}}{\tau_k}\mathcal{N}^s_{e}(u_{k},u_{k}^h; u_{k} - u^h_{k}, z_{k+1}) \Bigg\rbrace
\end{align*}
After applying summation by parts on $(z_{k+1}-z_k, u_k-u_k^h)$, i.e.,
\begin{multline*}
\sum_{k=1}^{N-1} \big(u_k-u_k^h,  z_{k+1}-z_k\big) =\\
 \left(u_N-u_N^h,z_N\right) - \left(u_1-u_1^h, z_1\right) - \sum_{k=1}^{N-1} \Big( z_{k+1},  \big(u_{k+1}-u_{k+1}^h\big)-\big(u_k-u_k^h\big)\Big)
\end{multline*}
it follows that
\begin{align*}
&\mcQ(I u_\tau) - \mcQ(I u_{\tau h}) \\
&\qquad = \left(u_1-u_1^h, z_1\right) + \tau_1 \Bigg\{ -\left( \frac{z_1-z_{0}}{\tau_1},u_0-u_0^h \right) - \mcN_e^s(u_0,u_0^h;u_0-u_0^h, z_1)\Bigg\}  \\
&\qquad \quad+ \tau_N  \Bigg\lbrace \mathcal{B}(z_{N}, u_{N} - u^h_{N}) + \mathcal{N}^s_{c}(u_{N},u_{N}^h; u_{N} - u^h_{N}, z_{N}) \Bigg\rbrace\\
&\qquad \quad+ \sum_{k=1}^{N-1} \tau_k  \Bigg\lbrace \bigg(\frac{u_{k+1} - u_k}{\tau_k}, z_{k+1}\bigg) -\bigg(\frac{u^h_{k+1} - u^h_k}{\tau_k}, z_{k+1} \bigg)   + \mathcal{B}(z_{k}, u_{k} - u^h_{k})\\
&\qquad \quad+ \mathcal{N}^s_{c}(u_{k},u_{k}^h;u_{k} - u^h_{k}, z_k) - \frac{\tau_{k+1}}{\tau_k}\mathcal{N}^s_{e}(u_{k},u_{k}^h; u_{k} - u^h_{k},z_{k+1}) \Bigg\rbrace
\end{align*}
Then, by shifting the indices of the arguments of $\mcB$ and $\mathcal{N}^s_{c}$:
\begin{align*}
&  \tau_{N}\mcB(z_{N}, u_{N} - u^h_{N})  + \sum_{k=1}^{N-1} \tau_{k}\mcB(z_{k}, u_{k} -  u^h_{k}) = \sum_{k=0}^{N-1} \tau_{k+1}\mcB(z_{k+1}, u_{k+1} - u^h_{k+1}) \\
& \tau_{N}\mathcal{N}^s_{c}(u_{N},u_{N}^h;u_{N} - u^h_{N},  z_{N}) + \sum_{k=1}^{N-1} \tau_{k}\mathcal{N}^s_{c}(u_{k},u_{k}^h; u_{k} - u^h_{k},  z_{k}) \\
& \qquad \qquad \qquad \qquad \qquad \qquad \qquad \qquad  =\sum_{k=0}^{N-1} \tau_{k+1}\mathcal{N}^s_{c}(u_{k+1},u_{k+1}^h; u_{k+1} - u^h_{k+1}, z_{k+1}) 
\end{align*}
and employing the mean-value linearization property \cref{eq:mean_l} on $\mcN_c^s$ and $\mcN_e^s$, we arrive at 
\begin{multline*}
\mcQ(I u_\tau) - \mcQ(I u_{\tau h}) = \left(u_0-u_0^h, z_{0}\right) \\
+ \sum_{k=0}^{N-1} \tau_{k+1}  \Bigg\lbrace  \bigg(\frac{u_{k+1} - u_k}{\tau_{k+1}}, z_{k+1} \bigg) + \mathcal{B}(u_{k+1}, z_{k+1}) + \mathcal{N}_{c}(u_{k+1};  z_{k+1})  -\mathcal{N}_{e}(u_{k}; z_{k+1}) \\ 
-\bigg(\frac{u^h_{k+1} - u^h_k}{\tau_{k+1}}, z_{k+1} \bigg) - \mathcal{B}(u^h_{k+1},z_{k+1})  - \mathcal{N}_{c}(u_{k+1}^h;  z_{k+1}) + \mathcal{N}_{e}(u_{k}^h; z_{k+1}) \Bigg\rbrace 
\end{multline*}
After substituting the time-discrete primal problem \cref{eq:p_time} weighted by dual solution $z_{k+1}$, we finally obtain
\begin{multline*}
\mcQ(I u_\tau) - \mcQ(I u_{\tau h}) = \left(u^0-u_0^h, z_{0}\right) + \sum_{k=0}^{N-1} \tau_{k+1}  \Bigg\lbrace  \left(f_{k+1}, z_{k+1}\right) -\bigg(\frac{u^h_{k+1} - u^h_k}{\tau_{k+1}}, z_{k+1} \bigg)  \\
- \mathcal{B}(u^h_{k+1},z_{k+1}) - \mathcal{N}_{c}(u_{k+1}^h;  z_{k+1}) + \mathcal{N}_{e}(u_{k}^h; z_{k+1}) \Bigg\rbrace.
\end{multline*}
This is \cref{eq:qoi_s} by the definition in \cref{eq:R_s}, and noting that $\forall v_{k+1}^h \in \mcS^{h,p}_{k+1}$ it holds that $\dual{r_{k+1}^h,v_{k+1}^h}=0$ by \cref{eq:p_full} for $k=0,1,\ldots, N-1$ and $\left(u^0-u_0^h, v_0^h\right)=0$ $\forall v_0^h \in \mcS^{h,p}_{0}$ by \cref{eq:u_0h}.
\end{proof}
\begin{remark}\label{remark_def_q}
Note that there is an alternative discrete representation of $\mcQ$ to the one in \cref{eq:qoi_discrete}:
$$
\mcQ(I w_{\tau}) = \ds{ \sum_{k=0}^{N-1} \tau_{k+1} ( q_k, w_k ) + ( q_N, w_N ) } 
$$
where 
$$
\left\lbrace \begin{array}{l} 
\ds{ q_k = \frac{1}{\tau_{k+1}} \int_0^T q \mbfN_k(t) \dd t \qquad \text{for } k=0,1,...,N-1 }\\
\ds{ q_N = \bar{q} + \int_0^T q \mbfN_N(t) \dd t }
\end{array} \right. 
$$ 
Instead of \cref{eq:dual_time0}-\cref{eq:dual_time2}, one then would expect an alternative time-stepping scheme for the dual problem in \cref{eq:d_wk} (solved backwards in time) as: find $z_k \in \mathcal{V}$, $k = 0,1, \ldots, N$, such that 
\begin{equation*}
-\left( \frac{z_{1} - z_0}{\tau_1}, w\right) - \mathcal{N}^s_{e}(u_{0},u_{0}^h; w, z_{1}) = ( q_{0}, w ) \qquad \forall w \in \mathcal{V}
\end{equation*} 
and for $k = 1,2, \ldots, N-1$:
\begin{multline*}
-\left( \frac{z_{k+1} - z_k}{\tau_{k+1}}, w\right) +  \frac{\tau_{k}}{\tau_{k+1}} \mathcal{ B}(z_k, w) + \frac{\tau_{k}}{\tau_{k+1}}\mathcal{N}^s_{c}(u_{k},u_{k}^h; w, z_k) \\
- \mathcal{N}^s_{e}(u_{k},u_{k}^h; w, z_{k+1}) = ( q_{k}, w )   \qquad \forall w \in \mathcal{V}
\end{multline*} 
where the terminal condition is 
\begin{equation*}
(z_N,w) + \tau_{N}\mcB(z_N,w) + \tau_{N}\mcN_c^s(u_N,u_N^h;w, z_N) = (q_N,w)  \qquad \forall w \in \mcV.
\end{equation*}  
However, this alternative time scheme is equivalent to \cref{eq:dual_time0}-\cref{eq:dual_time2}, and therefore leads to exactly the same time-discrete error representation \cref{eq:qoi_s}. 
\end{remark}
\subsection{Spatial and temporal error representation}
Building on Verf\"{u}rth's residual decomposition \cref{eq:r_dep} and the time-discrete error representation \cref{eq:qoi_s}, we are now ready to state our main result: A suitable decomposition of the dual-weighted residual \cref{eq:q_st_r}.
\begin{theorem}[Decomposed error representation]\label{thm:decomp_theorem}
Let the assumptions of \cref{thm:discrete_theorem} hold. Let $u$ denote the solution of the primal problem \cref{eq:p_wk} and $z$ denote the solution of the dual problem \cref{eq:d_wk}. 
Then the following error representation holds:
\begin{align}
\mcQ(u)-\mcQ(I u_{\tau h}) =&\mcR_\mathrm{s}(u_{\tau h}; z_{\tau}-v_{\tau}^h) + \mcR_\mathrm{t}(u_{\tau h}, u_{\tau}, z, z_{\tau}) + \operatorname{Osc} \label{eq:decomQ}
\end{align}
for any $v_{\tau}^h:=\{ v_{k}^h\}_{k=0}^{N}$,  $v_{k}^h \in \mcS^{h,p}_{k}$, where 
$\mcR_\mathrm{s}(u_{\tau h}; z_{\tau}-v_{\tau}^h)$ is the spatial error representation
\begin{equation}\label{eq:qoi_rs}
\mcR_\mathrm{s}(u_{\tau h}; z_{\tau}-v_{\tau}^h) :=\left(u^0-u_0^h, z_{0}-v_0^h\right) + \sum_{k=1}^{N} \tau_k \left\langle r_h^k , z_k- v_k^h\right\rangle,
\end{equation}
$\mcR_\mathrm{t}(u_{\tau h}, u_{\tau}, z, z_{\tau})$ is the temporal error representation
\begin{multline}\label{eq:qoi_t}
\mcR_\mathrm{t}(u_{\tau h}, u_{\tau}, z,z_{\tau}) :=  \left(u^0 - u_0^h, z(0) - z_{0}\right)  \\
+ \sum_{k=0}^{N-1} \int_{t_k}^{t_{k+1}}  \Big\lbrace \mcR_\mathrm{PDE}\big(I u_{\tau h}(t);z(t)-z_{k+1}\big)  + \left\langle r_{\tau}^{k+1}(t), z_{k+1} \right\rangle \Big\rbrace \, \dd t 
\end{multline}
and $\operatorname{Osc}$ denotes the data-oscillation contribution
\begin{equation}\label{eq:qoi_f}
\operatorname{Osc} := \sum_{k=0}^{N-1} \int_{t_k}^{t_{k+1}} \left\langle r_f^{k+1}(t),z_{k+1} \right\rangle \, \dd t .
\end{equation}
\end{theorem}
\begin{remark}
By virtue of Galerkin orthogonality,  $\mcR_\mathrm{s}(u_{\tau h}; z_{\tau})$ will vanish if $z_k \in \mcS^{h,p}_{k}$, $k=0,1,\ldots,N$. In addition, as $h \to 0^+$ it holds that $u_k^h \to u_k$ for all $k$ and, accordingly, $\mcR_\mathrm{s}(u_{\tau h}; \cdot)\to 0$ (see \cref{eq:p_time} and \cref{eq:u0}). Similarly, assuming sufficient smoothness in time, then $z_{k}, z_{k+1} \to z(t)$ for $t \in [t_k, t_{k+1}]$ as $\tau_{k+1} \to 0^+$, which implies $ \left(u^0 - u_0^h, z(0) - z_{0}\right) \to 0$ and $\mcR_\mathrm{PDE}(I u_{\tau h}(t);z(t)-z_{k+1}) \to 0$ as $\tau_{k+1} \to 0^+$. And since $r_{\tau}^{k+1} (t) \to 0$ as $\tau_{k+1} \to 0^+$ (see \cref{remark_residual}), we conclude that $\mcR_\mathrm{t}(u_{\tau h}, u_{\tau}, z,z_{\tau}) \to 0$ as $\tau_{k+1} \to 0^+$. This is the motivation for calling $\mcR_\mathrm{t}(u_{\tau h}, u_{\tau}, z,z_{\tau})$ and $\mcR_\mathrm{s}(u_{\tau h}; z_{\tau})$ the \emph{temporal error representation} and the \emph{spatial error representation}, respectively. 
\end{remark}
 \begin{remark}
 If we choose $f_{k+1} = \frac{1}{\tau_{k+1}}\int_{t_{k}}^{t_{k+1}} f(t) \, \dd t$, then the data-oscillation contribution~\cref{eq:qoi_f} will vanish.
 \end{remark}
\begin{proof}(of Theorem~\ref{thm:decomp_theorem})~
The global space-time error representation is \cref{thm:error_theorem} with $\hat{u}=I u_{\tau h}$:
\begin{align}
\mcQ(u)-\mcQ(I u_{\tau h}) = &\mathcal{R}_0\left( u_0^h;z(0)\right)  + \int_0^T  \mathcal{R}_\mathrm{PDE} \left( I u_{\tau h}(t); z(t)\right) \, \dd t.  \label{eq:qoi_st}
\end{align}
The spatial error representation \cref{eq:qoi_rs} satisfies the representation in \cref{thm:discrete_theorem}.\par
The temporal error representation is obtained by subtracting the spatial error representation \cref{eq:qoi_rs} and the data-oscillation contribution \cref{eq:qoi_f} from the space-time error representation \cref{eq:qoi_st}, i.e.,
\begin{align*}
\mcR_\mathrm{t}(u_{\tau h},  u_{\tau},z,z_{\tau}) = & \mcQ(u) -\mcQ(I u_{\tau h}) - \mcR_\mathrm{s}(u_{\tau h}; z_{\tau}) -  \operatorname{Osc} \\[5pt]
=&\mathcal{R}_0\left( u_0^h; z(0)\right) -\left( z_{0}, u^0-u_0^h\right) - \operatorname{Osc} \\
&+ \sum_{k=0}^{N-1} \int_{t_k}^{t_{k+1}}  \Big\lbrace  \mcR_\mathrm{PDE}\big(I u_{\tau h}(t);z(t)\big) - \left\langle r_h^{k+1}, z_{k+1} \right\rangle \Big\rbrace \, \dd t .
\end{align*}
Adding and subtracting $\mcR_\mathrm{PDE}(I u_{\tau h};z_{k+1})$ yields
\begin{multline*}
\mcR_\mathrm{t}(u_{\tau h}, u_{\tau},z,z_{\tau}) = \mathcal{R}_0\left( u_0^h; z(0)\right) -\left( z_{0}, u^0-u_0^h\right) -\operatorname{Osc} \\
+ \sum_{k=0}^{N-1} \int_{t_k}^{t_{k+1}} \! \Big\lbrace  \mcR_\mathrm{PDE}\big(I u_{\tau h}(t);z(t)-z_{k+1}\big) +\mcR_\mathrm{PDE}\left(I u_{\tau h}(t);z_{k+1}\right)- \left\langle r_h^{k+1}, z_{k+1} \right\rangle \Big\rbrace \, \dd t .
\end{multline*}
Since $\partial_t I u_{\tau h} =  \big(u^h_{k+1} - u^h_k\big)/\tau_{k+1}$ on $(t_k, t_{k+1}]$ according to the definition of $I u_{\tau h}$ \cref{eq:u_interp}, we employ the definition of the residuals in \cref{eq:R_st}, \cref{eq:R_0} and \cref{eq:R_s}, and obtain
\begin{align*} 
\mcR_\mathrm{t}(u_{\tau h},& u_{\tau},z, z_{\tau}) \\
&= \left( u^0-u_0^h, z(0)-z_{0}\right)  - \operatorname{Osc} + \sum_{k=0}^{N-1} \int_{t_k}^{t_{k+1}}  \Big\lbrace\dual{f(t),z_{k+1}} -\left(f_{k+1},z_{k+1}\right) \Big\rbrace \, \dd t \\
&\quad+ \sum_{k=0}^{N-1} \int_{t_k}^{t_{k+1}}  \Big\lbrace \mcR_\mathrm{PDE}\big(I u_{\tau h}(t);z(t)-z_{k+1}\big) - \mathcal{B} \left(I u_{\tau h}(t)-u^h_{k+1} , z_{k+1} \right) \\
&\quad- \mathcal{N}\left(I u_{\tau h}(t); z_{k+1}\right)  + \mathcal{N}_{c} \left(u_{k+1}^h; z_{k+1}\right) - \mathcal{N}_{e} \left(u^h_{k}; z_{k+1}\right) \Big\rbrace \, \dd t .
\end{align*}
Finally, substituting the definitions in \cref{eq:r_t}, \cref{eq:r_f} and \cref{eq:qoi_f} gives the result \cref{eq:qoi_t}.
\end{proof}
A useful interpretation of the spatial and the temporal error representation can be obtained by writing the global space-time error as:
\begin{equation}\label{eq:error}
\mcQ(u) - \mcQ(I u_{\tau h}) = \mcQ(u) - \mcQ(I u_{\tau}) + \mcQ(I u_{\tau}) - \mcQ(I u_{\tau h}) 
\end{equation}
The following Corollary holds:
\begin{corollary}
Under the assumptions of \cref{thm:discrete_theorem} and \cref{thm:decomp_theorem}, we have
\begin{equation} \label{eq:error_time}
\mcQ(u)- \mcQ(I u_{\tau}) = \mcR_\mathrm{t}(u_{\tau h}, u_{\tau},z, z_{\tau}) + \operatorname{Osc}
\end{equation}
and 
\begin{equation}\label{eq:error_space}
\mcQ(I u_{\tau})- \mcQ(I u_{\tau h}) = \mcR_\mathrm{s}(u_{\tau h}; z_{\tau} - v_{\tau}^h) 
\end{equation}
for any $v_{\tau}^h:=\{ v_{k}^h\}_{k=0}^{N}$,  $v_{k}^h \in \mcS^{h,p}_{k}$.
\end{corollary}
\begin{proof}
The identity in \cref{eq:error_space} is a direct consequence of \cref{eq:qoi_s} and \cref{eq:qoi_rs}. Equation \cref{eq:error_time} then follows from \cref{eq:decomQ} and \cref{eq:qoi_rs}.
\end{proof}

\section{Computable error estimate}\label{sec:computable}
There are two approximations commonly involved in evaluating the exact error representations \cref{eq:qoi_st}, \cref{eq:qoi_rs} and \cref{eq:qoi_t}:
\begin{itemize}
\item one for approximating the exact primal solutions $u$ and $u_{\tau}$ in the mean-value-linearized dual problem \cref{eq:d_wk} and its time-discrete system \cref{eq:dual_time0}-\cref{eq:dual_time2},
\item the other for approximating the exact dual solutions $z$ and $z_{\tau}$ in the error representation formulas \cref{eq:qoi_st}, \cref{eq:qoi_rs} and \cref{eq:qoi_t}.
\end{itemize}   
The resulting error estimate can only be accurate if the approximations are sufficiently close to the true solutions. \par
Here, to obtain computable and asymptotically effective error estimates, we consider a hierarchical two-level methodology developed in \cite{csimcsek2015duality} where the estimate is directly evaluated with an enriched dual approximation that is computed with help of an additional primal approximation at an enriched discretization level for the mean-value-linearization. Since the focus of \cite{csimcsek2015duality} is on the spatial discretization error, we now extend this methodology to our problem: We need two additional discretization levels: one which is spatially-enriched for evaluating the spatial error representation \cref{eq:qoi_rs} and the other which is space-time enriched for evaluating \cref{eq:qoi_t} and \cref{eq:qoi_st}.  \par
We first introduce the following notations:
\begin{remunerate}
\item[$\bullet$] $\mathcal{S}^{h,p}_{k}$: the original FE space with spatial mesh of size $h=h_k$ at time $t_{k}$ for $k=0,1,\ldots,N$.
\item[$\bullet$] $\mathcal{S}^{h/2,p}_k$: an \emph{enriched} FE space with finer spatial mesh of size $\ds{\frac{h_{k}}{2}}$ at time $t_k$ for $k=0,1,\ldots,N$. $\mathcal{S}^{h/2,p}_k$ is obtained by global refinement of all the element in $\mathcal{S}^{h,p}_k$. 
\item[$\bullet$] $\mathcal{S}^{h/2,p}_{k + 1/2}$: an enriched FE space with finer spatial mesh of size $\ds{\frac{h_{k}}{2}}$ at the intermediate time level $\ds{t_{k+1/2}=\frac{t_{k+1}+t_k }{2}}$ for $k=0,1,\ldots,N-1$. We set $\mathcal{S}^{h/2,p}_{k + 1/2}=\mathcal{S}^{h/2,p}_{k+1}$.
\item[$\bullet$] $u_{\tau, h}=\lbrace u_{k}^{h} \rbrace_{k=0}^{N}$: the solution of \cref{eq:p_full} and \cref{eq:u_0h} using time-step sizes $\ds{\{\tau_k\}_{k=1}^N}$ and FE spaces $\ds{\{\mcS^{h,p}_{k}\}_{k=0}^{N}}$. 
\item[$\bullet$] $u_{\tau, h/2}= \lbrace u_{k}^{h/2} \rbrace_{k=0}^{N}$: the solution of \cref{eq:p_full} and \cref{eq:u_0h} using time-step sizes $\ds{\{\tau_k\}_{k=1}^N}$ and enriched FE spaces $\ds{\{\mcS^{h/2,p}_{k}\}_{k=0}^{N}}$; $u_{\tau, h/2}$ represents an approximation of the time-discrete primal solution $u_{\tau}$.
\item[$\bullet$] $u_{\tau /2, h/2} = \lbrace \tilde{u}_{l}^{h/2} \rbrace_{l=0,1/2,\ldots, N}$: the solution of \cref{eq:p_full} and \cref{eq:u_0h} using \emph{half} time-step sizes $\ds{\left\lbrace \frac{\tau_1}{2},  \frac{\tau_1}{2},  \frac{\tau_2}{2},  \frac{\tau_2}{2}, \ldots,  \frac{\tau_N}{2},  \frac{\tau_N}{2}\right\rbrace}$ and enriched FE spaces $\{ \mathcal{S}^{h/2,p}_{l} \}_{l=0,1/2,\ldots, N}$; $u_{\tau /2, h/2}$ represents an approximation of the exact primal solution~$u$. 
\item[$\bullet$] $z_{\tau, h/2} = \lbrace z_{k}^{h/2} \rbrace_{k=0}^{N}$: the solution of the approximate dual problem obtained by replacing $u_{\tau}$ with $u_{\tau, h/2} $ in \cref{eq:dual_time0}-\cref{eq:dual_time2}, using time-step sizes $\ds{\{\tau_k\}_{k=1}^N}$ and enriched FE spaces $\ds{\{\mcS^{h/2,p}_{k}\}_{k=0}^{N}}$; $z_{\tau, h/2}$ represents an approximation of the time-discrete dual $z_{\tau}$.
\item[$\bullet$] $z_{\tau /2, h/2} = \lbrace \tilde{z}_{l}^{h/2} \rbrace_{l=0,1/2,\ldots,N}$: the solution of the approximate dual problem obtained by replacing $u_{\tau}$ with $u_{\tau /2, h/2} $ in \cref{eq:dual_time0}-\cref{eq:dual_time2}, using half time-step sizes $\ds{\left\lbrace \frac{\tau_1}{2},  \frac{\tau_1}{2},  \frac{\tau_2}{2},  \frac{\tau_2}{2}, \ldots,  \frac{\tau_N}{2},  \frac{\tau_N}{2}\right\rbrace}$ and enriched FE spaces $\{ \mathcal{S}^{h/2,p}_{l} \}_{l=0,1/2,\ldots, N}$; $z_{\tau /2, h/2}$ represents the approximation of the exact dual solution $z$. 
\end{remunerate}
The strategy for computing the primal and dual solutions is illustrated in \Cref{fig:computableE}. For evaluating the error representations \cref{eq:qoi_st}, \cref{eq:qoi_rs} and \cref{eq:qoi_t}, we compute two enriched dual solutions $z_{\tau ,h/2}$ and~$z_{\tau /2, h/2}$ solved backwards in time to approximate~$z_{\tau}$ and $z$, respectively. In order to make $z_{\tau ,h/2}$ computable, an additional primal approximation $u_{\tau, h/2}$ is computed forwards in time using FE spaces $\{\mcS^{h/2,p}_{k}\}_{k=0}^N$ to approximate the mean-value-linearization of the dual problem \cref{eq:dual_time0}-\cref{eq:dual_time2}. Similarly, another additional primal approximation $u_{\tau / 2,h/2}$ is computed using spaces $\{\mcS^{h/2,p}_{l}\}_{l=0,1/2,\ldots,N}$ for obtaining $z_{\tau /2,h/2}$. 
\begin{figure}[tbhp]
\begin{center}
\includegraphics[scale=0.22]{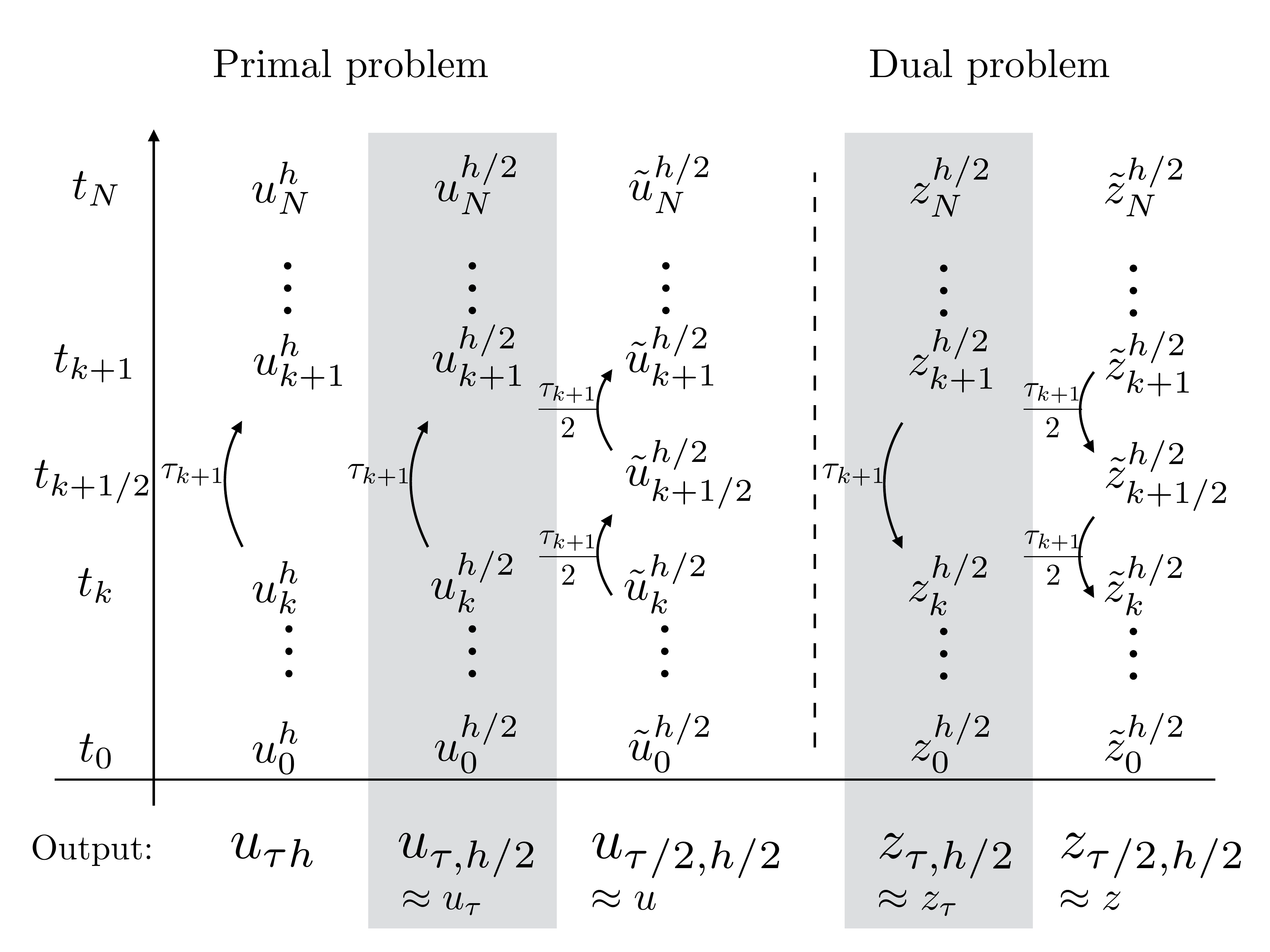} 
\end{center}
\caption{Approximations of the primal and dual solutions. The primal approximations $u_{\tau h}$, $u_{\tau,h/2}$ and $u_{\tau/2,h/2}$ are computed forwards in time and the dual approximations $z_{\tau, h/2}$, $z_{\tau,h/2}$ and $z_{\tau/2,h/2}$ are computed backwards in time with the corresponding spatial meshes and time steps. The computational cost of the algorithm can be reduced by discarding the approximations in the grey columns; see \cref{reduce_compute}.}
\label{fig:computableE}
\end{figure}
\par
Now let us denote by $\hat{z} \in \mcV$ a time-reconstruction of the dual solution $z_{\tau/2,h/2}$ (e.g. a piecewise-constant time-reconstruction will be used in numerical applications; see \Cref{sec:applications}). By replacing $z$ with $\hat{z}$ in \cref{eq:qoi_st}, the estimate of the space-time error in $\mcQ$ can then be computed as: 
\begin{equation}\label{eq:e_st}
\mcQ(u) - \mcQ(I u_{\tau h}) \approx \mathcal{E}_\mathrm{st} :=  \mathcal{R}_0\left( u_0^h; \hat{z}(0)\right) + \int_0^T  \mathcal{R}_\mathrm{PDE} \left( I u_{\tau h}(t); \hat{z}(t)\right) \, \dd t . 
\end{equation}
Replacing $z_{\tau}$ with the computable $z_{\tau, h/2}$ in \cref{eq:qoi_rs}, we compute the spatial error estimate $\mathcal{E}_\mathrm{s}$ as:   
\begin{align}\label{eq:e_s}
\mcR_\mathrm{s}(u_{\tau h}; z_{\tau}) \approx \mathcal{E}_\mathrm{s} & :=  \left(u_0^{h/2} - u_0^h, z_{0}^{h/2}\right) + \sum_{k=1}^{N} \tau_k \left\langle r_h^k, z_k^{h/2} \right\rangle.
\end{align}
Finally, replacing $z$ and $z_{\tau}$ with the computable $\hat{z}$ and $z_{\tau, h/2}$ in \cref{eq:qoi_t}, respectively,  we compute the temporal error estimate $\mathcal{E}_\mathrm{t}$ as: 
\begin{multline}\label{eq:e_t}
\mcR_\mathrm{t}(u_{\tau h}, u_{\tau},z, z_{\tau}) \approx \mathcal{E}_\mathrm{t}  :=   \left(u_0^{h/2} - u_0^h, \hat{z}(0) - z_{0}^{h/2}\right) \\
+ \sum_{k=0}^{N-1} \int_{t_k}^{t_{k+1}}  \bigg\lbrace  \mcR_\mathrm{PDE}\left(I u_{\tau h}(t) ;\hat{z}(t) -z_{k+1}^{h/2}\right)+ \left\langle r_{\tau}^{k+1}(t),z_{k+1}^{h/2} \right\rangle  \bigg\rbrace \, \dd t .
\end{multline}
\begin{remark}[Reduced-cost implementation]\label{reduce_compute}
If one wants to reduce the number of distinct approximations in the error estimates \cref{eq:e_st}--\cref{eq:e_t}, the most straightforward strategy is to simply take $u_{\tau, h/2}=u_{\tau/2, h/2}$ and $z_{\tau, h/2}=z_{\tau/2, h/2}$ at concurrent time steps (i.e. $u_k^{h/2} = \tilde{u}_k^{h/2}$ and $z_k^{h/2} = \tilde{z}_k^{h/2}$ for $k=0,1,\ldots,N$). In this manner, one only needs to compute $u_{\tau h}$, $u_{\tau/2, h/2}$ and $z_{\tau/2, h/2}$, without the gray columns in \Cref{fig:computableE}. More detailed description and examples are given in numerical applications; see \Cref{sec:applications}. An even cheaper alternative is to compute higher-order reconstructions using only $u_{\tau h}$ and $z_{\tau h}$; see, Becker and Rannacher \cite{becker1996weighted} for an overview, or coarse-scale adjoints~\cite{CarEstJohLarTavSISC2010}.
 \end{remark}

\section{Adaptive algorithm} \label{sec:adp_alg}    
Our goal is now to design an adaptive algorithm to iteratively increase the accuracy of the numerical solution by using the error estimates. In this section, we first derive error indicators of local contributions that serve as the basis to control adaptive mesh refinement and adaptive time-step selection, and then present the space-time adaptive algorithm. 
\subsection{Error indicators} 
To drive space-time adaptivity, the information of the global error estimates has to be localized to time-intervals and spatially-local contributions. To this end, we rewrite the computable error estimates $\mcE_\mathrm{st}$, $\mcE_s$ and $\mcE_\mathrm{t}$ in \cref{eq:e_st}--\cref{eq:e_t} as a sum of their local contributions on each time intervals $[t_k,t_{k+1}]$, $k=0,1,\ldots,N-1$, respectively. The absolute values of these local contributions are identified as the local indicators, which can directly be used for adaptive time-step selection. For adaptive mesh refinement, the local contributions associated to the spatial discretization error have to be localized further in space. We summarize the result in the following propositions. 
\begin{proposition} The error estimate $\mcE_\mathrm{st}$, $\mcE_\mathrm{t}$ and $\mcE_\mathrm{s}$ can be bounded from above by 
\begin{align*}
\vert  \mathcal{E}_\mathrm{st} \vert  \leq  \mathcal{E}^0_{h \tau} +\sum_{k=0}^{N-1}   \mathcal{E}^{k+1}_{h \tau} \qquad \;
\vert  \mathcal{E}_\mathrm{t} \vert  \leq \mcE^{0}_\tau + \sum_{k=0}^{N-1}  \mcE^{k+1}_\tau \qquad \;
\vert  \mathcal{E}_\mathrm{s} \vert  \leq  \mcE^{0}_h +\sum_{k=0}^{N-1}  \mcE^{k+1}_h  
\end{align*}
where the local space-time error indicators $\mcE^0_{h \tau}$ and $\mcE^{k+1}_{h \tau}$ are defined by
\begin{align} \label{eq:est_ht}
\mcE^0_{h \tau}  := \Big\vert \mathcal{R}_0\left( u_0^h; \hat{z}(0)\right) \Big\vert  \qquad \qquad \mcE^{k+1}_{h \tau}  := \bigg\vert \int_{t_k}^{t_{k+1}}  \mcR_\mathrm{PDE} \left( I u_{\tau h}(t); \hat{z}(t)\right) \, \dd t \bigg\vert ,
\end{align}
the temporal error indicators $\mcE^{0}_\tau$ and $\mcE^{k+1}_\tau$ are defined by
\begin{align}
&\mcE^{0}_\tau :=  \left\vert    \left(u_0^{h/2} - u_0^h, \hat{z}(0) - z_{0}^{h/2}\right)  \right\vert \label{eq:est_time0}\\
\begin{split}\label{eq:est_time}
&\mcE^{k+1}_\tau := \left\vert \int_{t_k}^{t_{k+1}}  \left\{  \mcR_\mathrm{PDE}\left(I u_{\tau h}(t) ;\hat{z}(t) -z_{k+1}^{h/2}\right) +\left\langle r_{\tau}^{k+1}(t), z_{k+1}^{h/2} \right\rangle \right\} \, \dd t \right\vert 
\end{split}
\end{align}
and the spatial error indicators $\mcE^{0}_h$ and $\mcE^{k+1}_h$ are defined by
\begin{align}\label{eq:est_step} 
\mcE^{0}_h := \Big\vert  \left(u_0^{h/2} - u_0^h, z_{0}^{h/2}\right)  \Big\vert \qquad  \qquad \mcE^{k+1}_h := \Big\vert \tau_{k+1} \dual{r_h^{k+1}, z_{k+1}^{h/2}} \Big \vert .              
\end{align}
\end{proposition}
\begin{proof}
We split the error estimator $\mcE_\mathrm{st}$ \cref{eq:e_st} into local space-time error indicators \cref{eq:est_ht} by 
\begin{align*}
\mathcal{E}_\mathrm{st} &=  \mcR_0\left( u_0^h; \hat{z}(0)\right) + \int_0^T  \mcR_\mathrm{PDE} \left( I u_{\tau h}(t); \hat{z}(t)\right) \, \dd t \\
&=  \mcR_0\left( u_0^h; \hat{z}(0)\right) +\sum_{k=0}^{N-1} \int_{t_k}^{t_{k+1}} \mcR_\mathrm{PDE} \left( I u_{\tau h}(t); \hat{z}(t)\right) \, \dd t  \\
&\leq \left\vert \mathcal{R}_0\left( u_0^h; \hat{z}(0)\right)  \right\vert+ \sum_{k=0}^{N-1} \left\vert \int_{t_k}^{t_{k+1}}  \mcR_\mathrm{PDE} \left( I u_{\tau h}(t); \hat{z}(t)\right) \, \dd t \right\vert 
\end{align*}
Following the same procedure as above, the temporal error estimate $\mathcal{E}_\mathrm{t}$ \cref{eq:e_t} is localized to the temporal error indicators \cref{eq:est_time} on each time intervals $[t_k,t_{k+1}]$, $k=0,1,\ldots,N-1$, and the spatial error estimate $\mathcal{E}_\mathrm{s}$ \cref{eq:e_s} is localized in time to the spatial error indicators \cref{eq:est_step}. 
\end{proof}
For spatial adaptivity we consider hierarchical mesh-refinement indicators as explained in~\cite[Section 4.2]{csimcsek2015duality}. In addition, let us note that instead of a traditional element-wise marking strategy, we use the function-support marking strategy introduced in \cite{Kuru2014fk} (see also \cite{Richter2015zr, Brummelen2017rr}).
\begin{remark}
In general in \cref{eq:est_step}, one needs to subtract an interpolant of the dual solution, $\Pi z_{\tau h/2}$ from $z_{\tau, h/2}$ to get a sharp spatial indicator. However, this is not needed for hierarchical indicators; see, \cite[Section 4.2]{csimcsek2015duality} for more details.
\end{remark}
\subsection{The space-time adaptive algorithm} 
In Algorithm 1, we propose a global space-time adaptive procedure using the above duality-based indicators. The pseudocode consists of three parts: (1) The computation of the primal and dual approximations (comprised of lines 3-8, 9-11 and 13-14), (2) the evaluation of the error estimates (given in lines 12, 15, 18-19 and 23), and (3) the error control (comprised of the remaining lines of Algorithm 1).
\begin{algorithm}[tbhp]
\caption{ Duality-based space-time adaptive algorithm }
\begin{footnotesize}
\begin{algorithmic}[1]
\Require Choose a coarse spatial mesh $\mathcal{K}_0$ and a coarse time step size $\tau$
\State Initialize a list of spatial mesh $\{\mathcal{K}_k\}_{k=1}^N$ ($\mathcal{K}_k=\mcK_0$) for time steps $\{ \tau_k\}_{k=1}^N$ ($\tau_k=\tau$)
\While{the maximal error estimate $Max$ $>$ tol}
  \For{$k \in \{1,2,\ldots,N\}$} 
  \State Compute $u_{\tau h}$ in $\mathcal{K}_{k}$ with $\tau_k$ 
  \State Compute $u_{\tau, h/2}$ in $\mathcal{K}^{h/2}_{k}$ with $\tau_k$ 
  \State Compute $u_{\tau/2, h/2}$ in $\mathcal{K}^{h/2}_{k}$ with $\tau_k/2$ 
  \State $t=t+\tau_k$
  \EndFor
  \For{$k \in \{N,N-1,\ldots,1\} $} 
        \State Compute $z_{\tau, h/2}$ in $\mathcal{K}^{h/2}_{k-1}$ with $\tau_k$ 
        \State Compute $z_{\tau/2, h/2}$ in $\mathcal{K}^{h/2}_{k-1}$ with $\tau_k/2$ 
        \State Estimate the error contribution $\mathcal{E}^{k}_{h\tau} $
        \State $t=t-\tau_k$
  \EndFor
  \State Estimate the initial error contribution $\mathcal{E}^{0}_{h\tau} $
  \State Compute the maximal error contribution for the whole time period $Max =\max\lbrace \mcE^0_{h\tau},\ldots,\mcE^N_{h\tau} \rbrace$
  \While{$\vert \mathcal{E}^k_{h\tau}\vert > \theta \, \vert Max \vert$}
    \State Estimate the local temporal error indicator $\mcE^k_{\tau}$
    \State Estimate the local spatial error indicator $\mcE^k_{h}$
    \If{$\mcE^k_{\tau} \geq \mcE^k_{h}$}
    \State Refine the time step $\tau_k$ by half
    \Else
    \State Estimate $\mcE^k_i $ for the mesh $\mathcal{K}_k$
    \State Refine the mesh $\mathcal{K}_k$ by using hierarchical refinement strategy and maximum strategy with parameter $\lambda$ 
    \EndIf
  \EndWhile
\EndWhile
\end{algorithmic}
\end{footnotesize}
\end{algorithm}
\par
Within the adaptive procedure, the error control is built on a two-step approach. First,  in lines 16-17 of the pseudocode we apply the maximum marking strategy (following Babu{\v{s}}ka and Vogelius \cite{bubuvska1984feedback}) with fraction $\theta \in [0,1]$ on space-time error indicators $\{\mathcal{E}^k_{h\tau}\}_{k=0}^{N}$ to globally select time steps $\{k'\}$ (i.e. $\vert \mcE^{k'}_{h\tau} \vert \geq \theta \, \max\{ \vert \mcE^0_{h\tau} \vert,\vert \mcE^1_{h\tau}\vert,\ldots,\vert\mcE^{N}_{h\tau}\vert\}$), which contain the largest error contributions throughout the time period. Second, in line 20 we locally check the leading causes of the error at the targeted time steps $\{k'\}$, whether from the spatial error or from the temporal error. An example of this is shown and explained in \Cref{fig:adaptivS}. The figure on the left indicates the time steps $\{k'\}$, where the major error contributions are located. Then, if the spatial indicator~$\mcE^{k'}_h$ is larger than the temporal indicator $\mcE^{k'}_\tau$, the spatial mesh is targeted for refinement according to the mesh indicators $\mcE_i^{k'}$; see \Cref{fig:adaptivS} (center). Otherwise, the time step size $\tau_{k'}$ is marked and reduced by half ; see \Cref{fig:adaptivS} (right).  
\par
The adaptive spatial mesh refinement is also based on a maximum marking strategy. The nodes $\{i'\}$ are marked for which their mesh-refinement indicators are at least a fraction $\lambda \in [0,1]$ of the maximal mesh indicator (i.e. $\vert \mcE_{i'}^{k} \vert \geq \lambda \max\{\vert\mcE_0^{k}\vert,\ldots, \vert\mcE_M^{k} \vert\}$). The addition of the basis function on selected nodes is performed using hierarchical refinement for finite element methods \cite{krysl2003natural,Kuru2014fk,Richter2015zr, Brummelen2017rr}. Moreover, instead of projection, we introduce a common refinement to transfer the solution from one mesh to another without loss of accuracy in any quadrature approximations. 
\begin{remark}
A standard adaptive algorithm for time-dependent problems starts with an initial coarse mesh, and proceeds \emph{sequentially}. Based on the mesh for the current time step, a new space mesh is generated for each new time step. Such a sequential procedure commonly uses residual-based error estimates in space which only contain information at the current time step. Duality-based error estimates, on the contrary, contain the entire evolution history of the error dependence implicitly via the dual solution. 
\end{remark}
\begin{figure}[tbhp]
\includegraphics[scale=0.12]{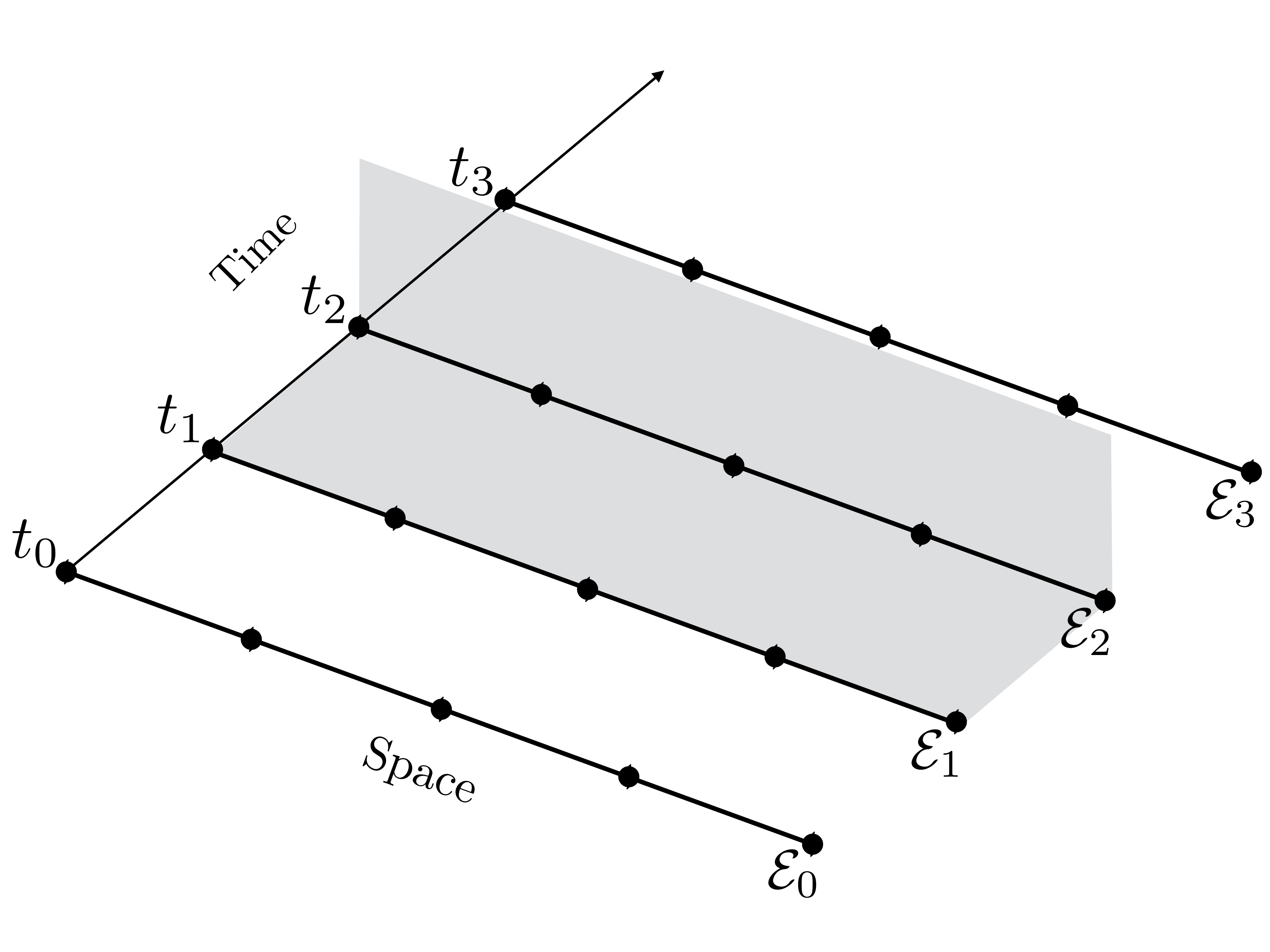}\includegraphics[scale=0.12]{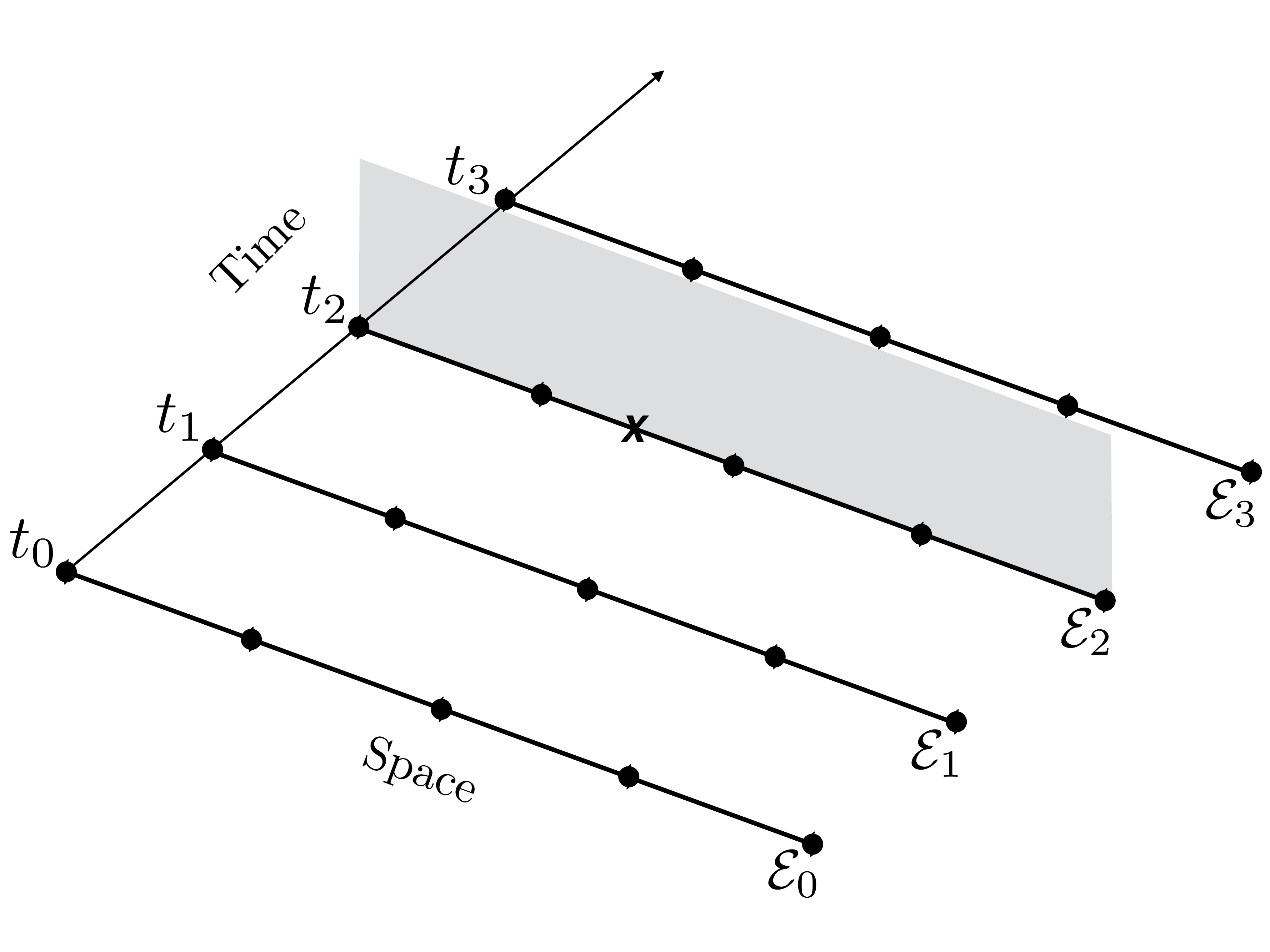}\includegraphics[scale=0.12]{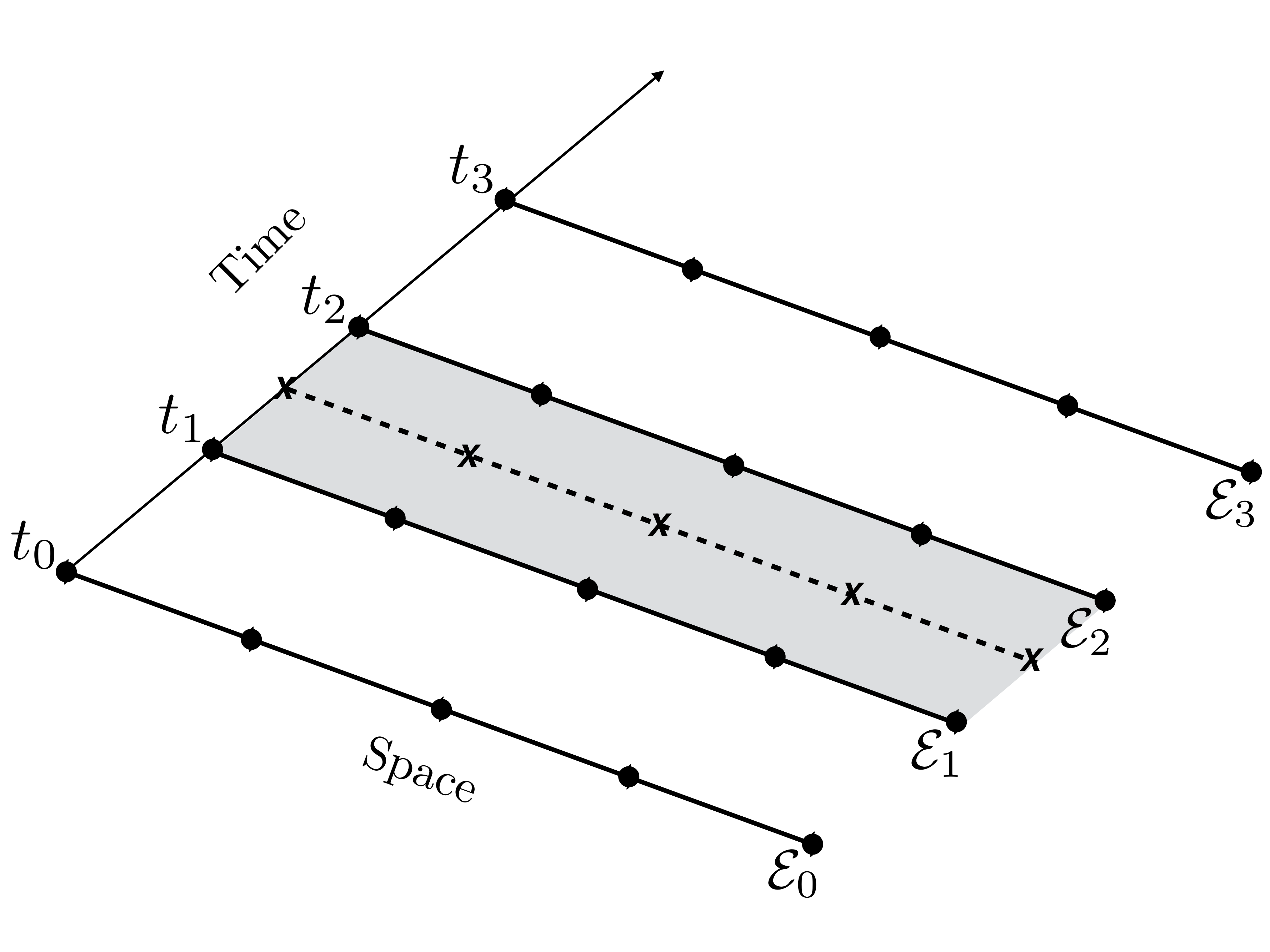}
\caption{Adaptive mesh refinement. Left: assume that $\mcE^2_{h\tau} = \max\{\mcE^0_{h\tau},\mcE^1_{h\tau},\mcE^2_{h\tau}, \mcE^3_{h\tau}\}$, i.e., the second time step $[t_1,t_2]$ is targeted as the largest error contribution throughout the time period. Middle: If the spatial indicator $\mcE^2_{h}$ is larger than the temporal indicator $\mcE^2_{\tau}$, the space mesh is targeted to refine according to the mesh indicators $\mcE^{2}_i$. Right: If the spatial indicator is smaller than the temporal indicator, the time step size is cut into half.}
\label{fig:adaptivS}
\end{figure}

\section{Applications} \label{sec:applications}
In this section, we give two examples of problems which fit into the abstract framework introduced in \cref{sec:abstract}: the nonlinear Allen--Cahn equation and the linear heat equation (as a special case of the Allen--Cahn equation). We numerically investigate the performance of the duality-based error estimates and the proposed adaptive algorithm. 
\par
Let us point out that the abstract framework easily accommodates other applications, for example, systems of parabolic equations; see~\cite[Section~6.3]{WuPhD2017} for the application to a phase-field tumor-growth system.
\subsection{Allen--Cahn equation} \label{sec:allen}
We subject the (forced) Allen--Cahn equation, $\pd_t u -\Delta u + \epsilon^{-2} \psi'(u) = f(t)$, to homogeneous Neumann boundary conditions. We choose the function spaces as $\mathcal{V}=H^1(\Omega)$, $\mathcal{V}^*=[H^{1}(\Omega)]^*$ and set $\mcB(u,v) = (\nabla u, \nabla v ) $, $\mcN(u;v) = \frac{1}{\epsilon^2} \left(\psi'(u), v\right)$ 
in \cref{eq:p_wk}, where $\epsilon$ is a parameter that controls the thickness of the diffuse interface (typical in phase-field models), and the nonlinear double-well function $\psi(u)$ is defined as (a standard truncated quartic polynomial)%
\footnote{Note that by choosing~$\psi(u) = 0$, one obtains the linear heat equation.}
\begin{equation}\label{eq:psi}
\psi(u) := \left\lbrace \begin{array}{ll}
(u+1)^2  & u < -1\\[8pt]
\ds{\frac{1}{4} (u^2-1)^2 }& u\in [-1,1]\\[8pt]
(u-1)^2 & u > 1.
\end{array} \right.
\end{equation}
Then, we obtain the weak form of the Allen--Cahn equation is: Find $u \in \mathcal{W}_{u^0}:= \left\{  v \in L^2(0,T; \mcV), \partial_t v \in L^2 \left( 0,T; \mcV^*\right) : v(0) = u^0 \right\}$ such that $\forall v \in L^2(0,T; H^1(\Omega))$
\begin{align}\label{eq:AC_wk}
\int_0^T \left( \left\langle  \partial_t u, v \right\rangle + (\nabla u, \nabla v ) + \frac{1}{\epsilon^2} \left(\psi'(u), v\right) \right) \dd t= 
\int_0^T \left\langle f,v\right\rangle \dd t.
\end{align}
In our setting the IMEX scheme for \cref{eq:AC_wk} leads to the energy-stable time-stepping scheme introduced in \cite{EyrPROC1998}: find $u_{k+1} \in H^1(\Omega)$ such that $\forall  v \in H^1(\Omega)$
\begin{align}\label{eq:AC_time}
\left( \frac{u_{k+1} - u_{k}}{\tau_{k+1}} , v \right) + (\nabla u_{k+1}, \nabla v ) + \frac{1}{\epsilon^2} \left(\psi'_c(u_{k+1}), v\right) - \frac{1}{\epsilon^2} \left(\psi'_e(u_{k}), v\right)= \left(\bar{f}_{k+1},v\right)
\end{align}
for $k=0,1,\ldots,N-1$, where the initial condition is $(u_0,v)=(u^0,v)$, $\forall v \in L^2(\Omega)$, and where we choose $\bar{f}_{k+1} = \frac{1}{\tau_{k+1}}\int_{t_{k}}^{t_{k+1}} f(\cdot, t) \, \dd t$. In particular, for a splitting of $\psi$ with a quadratic convex part, the resulting system is linear, for example:
$$
\psi = \psi_c - \psi_e = \left\lbrace \begin{array}{ll}
\ds{\left( u^2 + \frac{1}{4} \right) - \left(-2u - \frac{3}{4}\right)}& u < -1\\
\ds{\left( u^2 + \frac{1}{4} \right) - \left( \frac{3}{2} u^2 - \frac{1}{4} u^4 \right)}& u\in [-1,1]\\
\ds{\left( u^2 + \frac{1}{4} \right) - \left(2u - \frac{3}{4}\right)}& u > 1 .
\end{array} \right. 
$$
We then have the full discretization: find $u_{k+1}^h \in \mcS^{h,1}_{k+1}$ such that $ \forall  v^h \in \mathcal{S}^{h,1}_{k+1}$
\begin{multline}\label{eq:AC_full}
 \!\left( \frac{u_{k+1}^h - u_{k}^h}{\tau_{k+1}} , v^h \right) + (\nabla u_{k+1}^h, \nabla v^h ) 
\\ 
 + \frac{1}{\epsilon^2} \left(\psi'_c(u^h_{k+1}), v^h\right) - \frac{1}{\epsilon^2} \left(\psi'_e(u^h_{k}), v^h\right)= \left(\bar{f}_{k+1},v^h\right)
\end{multline}
for $k=0,1,\ldots,N-1$, where the initial condition is $(u^h_0,v^h)=(u^0,v^h)$, $\forall v^h \in \mathcal{S}^{h,1}_0$.\par
According to the definition of $\mcN^s$ in \cref{eq:no}, we can explicitly write the mean-value linearization of $\psi'(u)$ in terms of $u$ and $\hat{u}$:
\begin{align*}
\psi'^s(u,\hat{u}) = \int_0^1 \psi''(su + (1-s)\hat{u} ) \dd s,
\end{align*}
although, because of its piecewise definition \cref{eq:psi}, this is an elaborate expression. For example, for $u,\hat{u} > 1$ or $u,\hat{u} < -1$, we have $\psi'^s(u,\hat{u}) = 2$, and for $u,\hat{u} \in [-1,1]$, we have $\psi'^s(u,\hat{u}) = u^2 + \hat{u}^2 + u \hat{u} -1$. \par
Then, by setting $\mcN^s(u,\hat{u};w, z) = \frac{1}{\epsilon^2} \big(\psi'^s(u,\hat{u})z, w\big)$ in \cref{eq:d_wk}, the dual problem reads: find $z \in \mathcal{W}^{\bar{q}}:= \left\{  v \in L^2(0,T; \mcV), \partial_t v \in L^2 \left( 0,T; \mcV^* \right) : v(T)= \bar{q} \right\}$ such that $\forall w \in L^2(0,T; H^1(\Omega))$
\begin{align}\label{eq:AC_dual}
\int_0^T \Big( \left\langle  -\partial_t z, w \right\rangle + (\nabla z, \nabla w ) + \frac{1}{\epsilon^2} \left(\psi'^s(u,\hat{u})z, w\right) \Big) \, \dd t=  \int_0^T (q, w) \, \dd t,
\end{align}
And the IMEX time-discrete dual problem, based on \cref{eq:dual_time0}-\cref{eq:dual_time2}, is defined by: find $z_k \in H^1(\Omega)$, $k=0,1,\ldots, N$, such that 
\begin{align}\label{eq:AC_dualT0}
-\left( \frac{z_{1} - z_{0}}{\tau_1}, w\right) - \frac{1}{\epsilon^2} \left(\psi'^s_e(u_{0}, u_{0}^h)z_{1}, w\right)= ( q_{0}, w ) \quad \forall w \in H^1(\Omega)
\end{align}
and for $k = 1,2,\ldots, N-1$:
\begin{multline}\label{eq:AC_dualT1}
-\Big( \frac{z_{k+1} - z_k}{\tau_k},  w\Big) + (\nabla z_k, \nabla w) + \frac{1}{\epsilon^2} \left(\psi'^s_c(u_{k}, u_{k}^h)z_k, w\right) \\
- \frac{\tau_{k+1}}{\tau_k} \frac{1}{\epsilon^2} \left(\psi'^s_e(u_{k}, u_{k}^h)z_{k+1}, w\right)= ( q_{k}, w ) \qquad \forall w \in H^1(\Omega)
\end{multline} 
where the terminal condition is 
\begin{multline}\label{eq:AC_dualT2} 
\left( z_{N} , w\right) + \tau_N (\nabla z_{N}, \nabla w) + \tau_N \frac{1}{\epsilon^2} \left(\psi'^s_c(u_{N}, u_{N}^h)z_{N}, w\right) \\
= (\bar{q},w) + \tau_N(q_N,w) \qquad \forall w \in H^1(\Omega)
\end{multline} 
with~$q_k$, $k=0,1,\ldots,N$ defined in Lemma~\ref{lem:lemma_q}. Note that for our choice of $\psi_c$, the derivative $\psi'^s_c$ reduces to a constant. The main results of \Cref{sec:section_dep} hold, as shown in the following corollary.
\begin{corollary}[Decomposed error representation for Allen--Cahn equation]
The following error representation holds
\begin{align*}
\begin{split}
\mcQ(u)-\mcQ(I u_{\tau h})  &= \left(z(0), u^0 - u_0^h\right) + \sum_{k=0}^{N-1}  \int_{t_k}^{t_{k+1}} \bigg\lbrace 
\left\langle f(t),z(t)\right\rangle
- \big( \partial_t I u_{\tau h}(t), z(t) \big) \\
& \quad - \big(\nabla I u_{\tau h}(t), \nabla z(t)\big) -\frac{1}{\epsilon^2} \big(\psi'(I u_{\tau h}(t)), z(t)\big) \bigg\rbrace \, \dd t
\end{split}\\
&= \mcR_\mathrm{s}(u_{\tau h}; z_{\tau}) + \mcR_\mathrm{t}(u_{\tau h}, u_{\tau}, z, z_{\tau}),
\end{align*}
where the spatial error representation reduces to
\begin{multline*}\label{eq:AC_qs}
 \mcR_\mathrm{s}(u_{\tau h}; z_{\tau}) =  (z_{0},u^0-u_0^h) +\sum_{k=0}^{N-1} \tau_{k+1} \left\lbrace \left( \bar{f}_{k+1} ,z_{k+1} \right) -  \left( \frac{u^h_{k+1} - u^h_{k}}{\tau_{k+1}} , z_{k+1} \right)\right.\\
\left. -  ( \nabla z_{k+1}, \nabla u_{k+1}^h) - \frac{1}{\epsilon^2} \left(\psi'_c(u_{k+1}^h), z_{k+1}\right)+ \frac{1}{\epsilon^2} \left(\psi'_e(u_{k}^h), z_{k+1}\right) \right\rbrace 
\end{multline*}
and the temporal error representation reduces to
\begin{equation*}\label{eq:AC_qt}
\begin{split}
& \mcR_\mathrm{t}(u_{\tau h}, u_{\tau},z, z_{\tau}) = \left(z(0) - z_{0}, u^0 - u_0^h\right) + \sum_{k=0}^{N-1} \int_{t_k}^{t_{k+1}}  \bigg\lbrace \left\langle f(t),z(t)-z_{k+1}\right\rangle
\\ 
&\quad - \left( \frac{u^h_{k+1} - u^h_{k}}{\tau_{k+1}}, z(t)-z_{k+1} \right)  
- (\nabla I u_{\tau h}(t), \nabla z(t)) + ( \nabla u_{k+1}^h, \nabla z_{k+1}) 
\\
&\quad - \frac{1}{\epsilon^2} \big(\psi'(I u_{\tau h}(t)), z(t)\big)  
+ \frac{1}{\epsilon^2} \left(\psi'_c(u_{k+1}^h), z_{k+1}\right) - \frac{1}{\epsilon^2} \left(\psi'_e(u_{k}^h), z_{k+1}\right) \bigg\rbrace \, \dd t .
 \end{split}
\end{equation*}
\end{corollary}
\begin{proof}
The result simply follows from \cref{thm:decomp_theorem} applied to \cref{eq:AC_wk}--\cref{eq:AC_dualT2}. 
\end{proof}
\subsection{Computable error indicator}
Let $u_{\tau h} = \lbrace u_k^{h/2} \rbrace_{k=0}^{N}$ denote the solution of the fully-discrete primal problem \cref{eq:AC_full} using time-step sizes $\{ \tau_k\}_{k=1}^N$ and FE spaces $\{ \mathcal{S}^{h,1}_{k} \}_{k=0}^{N}$, and let $u_{\tau, h/2} = \lbrace u_k^{h/2} \rbrace_{k=0}^{N}$ denote the solution of \cref{eq:AC_full} using time-step sizes $\{ \tau_k\}_{k=1}^N$ and enriched FE spaces $\{ \mathcal{S}^{h/2,1}_{l} \}_{l=0,1/2,\ldots, N}$. Replacing $u_{\tau}$ with the computable $u_{\tau, h/2}$ in \cref{eq:AC_dualT0}-\cref{eq:AC_dualT2}, we obtain the full discretization of the dual problem using enriched FE spaces: find $z_{k}^{h/2} \in \mcS^{h/2,1}_k$, $k=0,1,\ldots, N$, such that 
\begin{equation}\label{eq:AC_dualF0}
-\left( \frac{z_{1}^{h/2}  - z_{0}^{h/2} }{\tau_1}, w^{h/2} \right) - \frac{1}{\epsilon^2} \left(\psi'^s_e(u_{0}^{h/2} , u_{0}^h)z_{1}^{h/2} , w^{h/2} \right)= ( q_{0}, w^{h/2}  ) \quad \forall  w^{h/2} \in \mathcal{S}^{h/2,1}_0
\end{equation}
and for $k = 1,2,\ldots, N-1$:
\begin{multline}\label{eq:AC_dualF1}
-\Bigg( \frac{z_{k+1}^{h/2} - z_k^{h/2}}{\tau_k} , w^{h/2}\Bigg) + (\nabla z_k^{h/2}, \nabla w^{h/2}) + \frac{1}{\epsilon^2} \left(\psi'^s_c(u_{k}^{h/2}, u_{k}^h)z_k^{h/2}, w^{h/2}\right) \\
- \frac{\tau_{k+1}}{\tau_k} \frac{1}{\epsilon^2} \left(\psi'^s_e(u_{k}^{h/2}, u_{k}^h)z_{k+1}, w^{h/2}\right)= ( q_{k}, w^{h/2} ) \qquad \forall  w^{h/2} \in \mathcal{S}^{h/2,1}_k
\end{multline} 
where the terminal condition is 
\begin{multline}\label{eq:AC_dualF2}
\left( z_{N}^{h/2} , w^{h/2} \right) + \tau_N(\nabla z_{N-1}^{h/2}, \nabla w^{h/2}) + \frac{\tau_N}{\epsilon^2} \left(\psi'^s_c(u_{N}^{h/2}, u_{N}^h)z_{N}, w^{h/2}\right) \\
= \tau_N( q_{N}, w^{h/2} ) +(\bar{q}, w^{h/2}) \qquad \forall  w^{h/2} \in \mathcal{S}^{h/2,1}_N
\end{multline} 
We denote by $z_{\tau,h/2}=\lbrace z_k^{h/2} \rbrace_{k=0}^{N}$ the solution of \cref{eq:AC_dualF0}-\cref{eq:AC_dualF2}. To get $z_{\tau/2,h/2}$, we first compute the space-time enriched approximation $u_{\tau/2,h/2}=\lbrace \tilde{u}_l^{h/2}\rbrace_{l=0,1/2,\ldots,N}$ of the primal problem using half time-step sizes $\{ \tau_1/2,  \tau_1/2,  \tau_2/2,  \tau_2/2, \ldots, \tau_N/2, \tau_N/2\}$ and enriched FE spaces $\{ \mathcal{S}^{h/2,1}_{l} \}_{l=0,1/2,\ldots, N}$. Then, replacing $u_{\tau}$ with $u_{\tau/2,h/2}$ in \cref{eq:AC_dualT0}-\cref{eq:AC_dualT2}, we compute the space-time enriched approximation $z_{\tau/2,h/2}=\lbrace \tilde{z}_l^{h/2}\rbrace_{l=0,1/2,\ldots,N}$ using the same time-step sizes and FE spaces as $u_{\tau/2,h/2}$. \par
In the numerical examples for the Allen--Cahn equation in \Cref{sec:num_AC} we compute $u_{\tau, h/2}$ and $z_{\tau, h/2}$ directly by taking $u_{\tau, h/2} = u_{\tau/2, h/2}$ and $z_{\tau, h/2} = z_{\tau/2, h/2}$ at concurrent time steps as in \cref{reduce_compute}, and consider a piecewise-constant time-reconstruction $\hat{z}$ of $z_{\tau/2, h/2}$ for each time interval $[t_k,t_{k+1})$, $k=0,1,\ldots,N-1$, i.e.
$$
\hat{z}(t) = \tilde{z}_k^{h/2} \quad \text{for } t \in [t_k,t_{k+1/2}), \qquad \hat{z}(t) = \tilde{z}_{k+1/2}^{h/2} \quad \text{for } t \in [t_{k+1/2},t_{k+1}).
$$
According to \cref{eq:e_st} and \cref{eq:est_ht}-\cref{eq:est_step}, we then get the following global space-time error estimate:
\begin{multline}\label{eq:ac_est_st}
\mcE_\mathrm{st} = \left(\hat{z}(0), u_0^{h/2} - u_0^h\right) + \sum_{k=0}^{N-1}  \int_{t_k}^{t_{k+1}} \left\lbrace \left\langle f(t),\hat{z}(t)\right\rangle - \left( \frac{u^h_{k+1} - u^h_{k}}{\tau_{k+1}}, \hat{z}(t) \right) \right.\\
\left.  - \big(\nabla I u_{\tau h}(t), \nabla \hat{z}(t)\big)  - \frac{1}{\epsilon^2} \big(\psi'(I u_{\tau h}(t)), \hat{z}(t)\big) \right\rbrace \, \dd t 
\end{multline}
and the local error indicators
\begin{align*}
\mcE^{k+1}_{h \tau} \!=&\Bigg\vert \int_{t_k}^{t_{k+1}}\! \Bigg\lbrace \left\langle f(t),\hat{z}(t)\right\rangle
- \left( \frac{u^h_{k+1} - u^h_{k}}{\tau_{k+1}}, \hat{z}(t) \right) - (\nabla I u_{\tau h}(t), \nabla \hat{z}(t)) \\
&- \frac{1}{\epsilon^2} \big(\psi'(I u_{\tau h}(t)), \hat{z}(t)\big) \Bigg\rbrace \, \dd t  \Bigg\vert \\
\mcE_{\tau}^{k+1} \!=&\Bigg\vert \int_{t_k}^{t_{k+1}} \! \Bigg\lbrace 
\left\langle f(t),\hat{z}(t) -\tilde{z}^{h/2}_{k+1}\right\rangle 
\\
&- \left( \frac{u^h_{k+1} - u^h_{k}}{\tau_{k+1}}, \hat{z}(t)-\tilde{z}_{k+1}^{h/2} \right) - \big(\nabla I u_{\tau h}(t), \nabla \hat{z}(t)\big) + (\nabla u_{k+1}^h, \nabla \tilde{z}_{k+1}^{h/2}) \\ 
&- \frac{1}{\epsilon^2} \big(\psi'(I u_{\tau h}(t)), \hat{z}(t)\big) + \frac{1}{\epsilon^2} \left(\psi'_c(u_{k+1}^h), \tilde{z}_{k+1}^{h/2}\right) - \frac{1}{\epsilon^2} \left(\psi'_e(u_{k}^h), \tilde{z}_{k+1}^{h/2}\right) \Bigg\rbrace \, \dd t\Bigg\vert 
\\
\mcE_h^{k+1}  \!=& \Bigg\vert \tau_{k+1} \Bigg\lbrace 
\left(\bar{f}_{k+1},\tilde{z}^{h/2}_{k+1}\right)
-  \left( \frac{u^h_{k+1} - u^h_{k}}{\tau_{k+1}} , \tilde{z}_{k+1}^{h/2} \right) -  ( \nabla \tilde{z}_{k+1}^{h/2}, \nabla u_{k+1}^h)  
\\ &-\frac{1}{\epsilon^2} \big(\psi'_c(u_{k+1}^h), \tilde{z}_{k+1}^{h/2} \big) 
 + \frac{1}{\epsilon^2} \big(\psi'_e(u_{k}^h), \tilde{z}_{k+1}^{h/2}\big) \Bigg\rbrace\Bigg\vert 
\end{align*}
for $k=0,1,\ldots,N-1$ and $i=0,1,\ldots, M_k$. And since $z_0^{h/2} =\hat{z}(0) = \tilde{z}_0^{h/2}$,
\begin{align*}
&\mcE_{h\tau}^0 = \mcE_h^0 = \left( u_0^{h/2} - u_0^{h}, \tilde{z}_0^{h/2}\right)\\
&\mcE_{\tau}^0 = 0.
\end{align*}
\begin{remark}
For the linear heat equation, $\partial_t u - \Delta u = f$, note that the dual problem and the computable error indicators are easily obtained from the above results by neglecting the nonlinear terms $\psi'_e$ and $\psi'_c$.   
\end{remark}
\subsection{Numerical results} \label{sec:num_AC}
In the following numerical experiments, we investigate the efficiency of the duality-based error estimates and the performance of the proposed adaptive algorithm. The results will be demonstrated in three parts. In the first part, we illustrate the consistency of the dual time scheme \cref{eq:AC_dualT0}--\cref{eq:AC_dualT2} since we introduced a non-standard IMEX time-discrete dual problem \cref{eq:dual_time0}--\cref{eq:dual_time2} which contains a nonstandard coefficient $\tau_{k+1}/\tau_k$. The second part is on the convergence of error estimate $\mcE_\mathrm{st}$ under uniform refinements for the nonlinear Allen--Cahn equation. In the third part, we apply the proposed duality-based adaptive algorithm to the linear heat equation and the nonlinear Allen--Cahn equation. We compare our adaptive result for the heat equation with the sequential-in-time adaptive algorithm (specifically for the heat equation) of Verf{\"u}rth stated in \cite[Sec. 6.8]{verfurth2013posteriori}. For the Allen--Cahn equation, we compare our adaptive results with uniform space-time refinements. 
\subsubsection*{Consistency test}
We begin with verifying numerically that the special IMEX time scheme of the dual problem \cref{eq:dual_time0}-\cref{eq:dual_time2} is first-order accurate in time with respect to refinements of uniform initial time steps (i.e. $\tau_{k+1}/\tau_k=1$) and, in particular, nonuniform initial time steps (i.e. $\tau_{k+1}/\tau_k \neq 1$). Here, the Allen--Cahn equation is considered in 1D on the domain $\Omega = (-3,3)$ with parameter $\epsilon = 1$. The spatial mesh is composed of $256$ elements along the axis. We consider a manufactured solution which oscillates in time:
$$
z(\mathbf{x},t) = e^{-10 x^2 + \sin{t}}
$$
The convergence results are presented on a double logarithmic scale in \Cref{fig:ac_consist}. \Cref{fig:ac_consist}a is the convergence of the error in $L^2$ norm at time $T$ based on time-step refinements using uniform initial time steps $\{\tau_k\}_{k=1}^5=\{0.1, 0.1, 0.1, 0.1, 0.1\}$, and \Cref{fig:ac_consist}b uses nonuniform initial time steps $\{\tau_k\}_{k=1}^5 =\{ 2\cdot 10^{-2} (3k +1)\}_{k=4}^0 =\{0.26, 0.2, 0.14, 0.08, 0.02\}$. For both uniform and nonuniform initial time steps, the observed rates are close to 1, which demonstrates that \cref{eq:dual_time0}-\cref{eq:dual_time2} is a first-order time-accurate scheme.
\begin{figure}[t]
\centering
\subfloat[Uniform]{\includegraphics[scale=0.33]{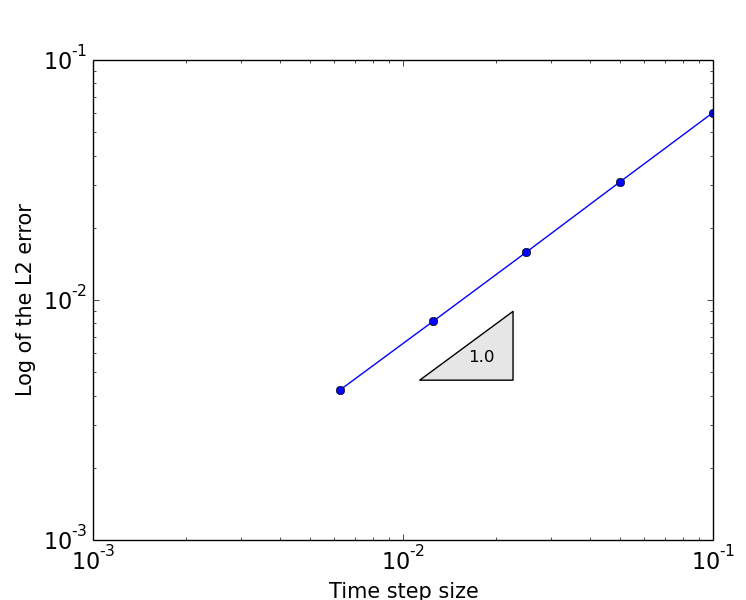} }
\subfloat[Nonuniform]{\includegraphics[scale=0.33]{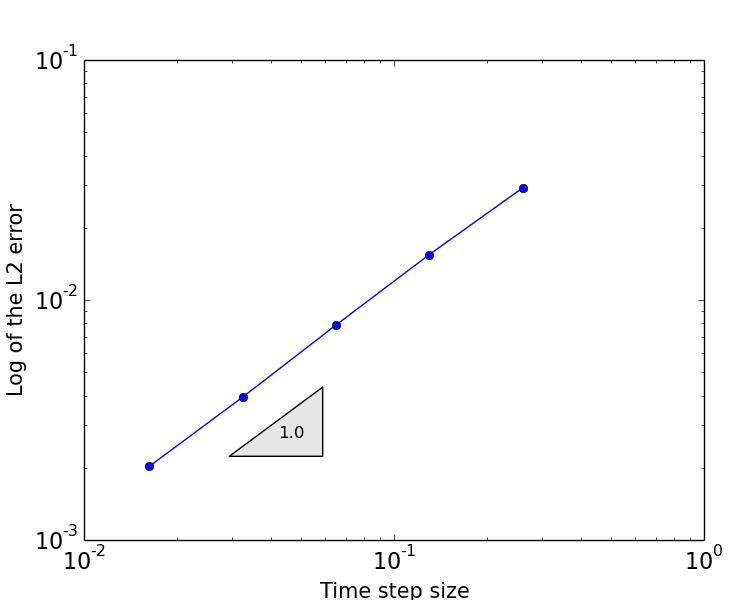} }
\caption{{\small Accuracy test for dual IMEX scheme: (a) log of the $L^2$ norm of the error at time $T$ versus time step size (uniform initial time steps); (b) log of the $L^2$ norm of the error at time $T$ versus maximum time step size (nonuniform initial time steps)}}\label{fig:ac_consist}
\end{figure}
\subsubsection*{Effectivity test}
In this numerical experiment, we consider the Allen--Cahn equation with an exact solution
$$
u(\mathbf{x},t) = \sin(\pi x ) \, \sin(\pi  y) \, e^{-t}
$$
on the domain $\Omega = (0,1)\times (0,1)$ where $\epsilon = 1$. We suppose that we are interested in the error at final time~$T$, i.e., in~\eqref{eq:qoi} we take $q=0$ and $\bar{q} = u(T) - I u_{\tau h}(T)$ (which is approximated in the computations by $\bar{q} \approx \tilde{u}_N^{h/2} - u_N^h$, see Section~\ref{sec:computable}).
The error of interest is thus $\mcQ(u) - \mcQ(I u_{\tau h})=\Vert u(T) - I u_{\tau h}(T) \Vert_{L^2(\Omega)}^2 = \Vert u(T) - u_N^h \Vert_{L^2(\Omega)}^2$. \par
To investigate the error estimate with respect to the spatial discretization, temporal discretization and space-time discretization, we compute $\mcE_\mathrm{st}$ according to \cref{eq:ac_est_st} under uniform spatial refinement, uniform temporal refinement and uniform space-time refinement, respectively. In view of the nonstandard time discretization scheme for the dual problem, we also investigate the accuracy of the error estimate $\mcE_\mathrm{st}$ with respect to uniform and nonuniform initial time step sizes. \Cref{tab:ac_eff_table_s} presents the convergence of $\mcE_\mathrm{st}$ under uniform spatial refinement for a sufficiently small time step size ($\tau_k = 1e-4$). For uniform initial time steps $\{\tau_k\}_{k=1}^{4} = \{ 0.05, 0.05, 0.05, 0.05\}$, the left side of \Cref{tab:ac_eff_table_t_u} shows the convergence of $\mcE_\mathrm{st}$ under uniform temporal refinement for a sufficiently fine spatial mesh ($128\times128$ elements) and the left side of \Cref{tab:ac_eff_table_n} shows the convergence under uniform space-time refinements. For nonuniform initial time steps $\{\tau_k\}_{k=1}^{4} = \{ 0.08, 0.06, 0.04, 0.02\}$, the convergence result for uniform temporal refinement with a sufficiently fine spatial mesh is presented on the right side of \Cref{tab:ac_eff_table_t_u}, and the convergence result for uniform space-time refinements is presented on the right side of \Cref{tab:ac_eff_table_n}. The effectivity results for different refinements are also presented in the two tables. \par 
\begin{table}[tbhp]
\centering
\begin{small}
\begin{tabular}[b]{l|l|l|l}
      $M$   & $\mcQ(u)-\mcQ(I u_{\tau h})$ & $\mcE_\mathrm{st}$&  Effectivity\\   
      \hline 
      16  &  0.0003046  &0.0002871&  0.943 \\
      64       & 1.711e-05 &1.606e-05&0.939 \\
      256      &1.041e-06 &9.756e-07&0.938 \\
     1024   &  6.458e-08  &6.054e-08&0.937 \\
     4096 &   4.028e-09   &3.776e-09& 0.937 \\
     \hline
\end{tabular}
\end{small}
\caption{{\small Effectivity of error estimate \cref{eq:ac_est_st} under spatial refinement}}\label{tab:ac_eff_table_s}
\end{table}

\begin{table}[tbhp]
\centering
\begin{small}
\begin{tabular}[b]{l|l|l|l|l|l|l}
      \multicolumn{1}{c}{ } &  \multicolumn{3}{c|}{Uniform initial time steps} & \multicolumn{3}{c}{Nonuniform initial time steps}\\
      \hline
      $N$  &  $\mcQ(u)-\mcQ(I u_{\tau h})$ & $\mcE_\mathrm{st}$&   Eff & $\mcQ(u) - \mcQ(I u_{\tau h})$ & $\mcE_\mathrm{st}$ &  Eff \\   
      \hline 
       4  & 0.0001378  &5.893e-05&  0.428 & 0.0001058  &4.773e-05& 0.451\\
      8  & 3.584e-05 &1.659e-05&  0.463 &2.612e-05 &1.253e-05&  0.480\\   
     16    &8.918e-06 &4.272e-06& 0.479  & 6.247e-06&3.054e-06& 0.489\\   
     32    & 2.130e-06&1.032e-06&0.484  & 1.444e-06 &7.063e-07&0.489\\
        64    &4.778e-07 &2.303e-07&0.482 & 3.122e-07 &1.507e-07&0.483\\
     \hline
\end{tabular}
\end{small}
\caption{{\small Effectivity of error estimate \cref{eq:ac_est_st} under temporal refinement with uniform initial time steps and nonuniform initial time steps}}\label{tab:ac_eff_table_t_u}
\end{table}


\begin{table}[tbhp]
\centering
\begin{small}
\begin{tabular}[b]{l|l|l|l|l|l|l}
      \multicolumn{1}{c}{ } &  \multicolumn{3}{c|}{Uniform initial time steps} & \multicolumn{3}{c}{Nonuniform initial time steps}\\
      \hline
      $M \times N $    & $\mcQ(u) - \mcQ(I u_{\tau h})$ & $\mcE_\mathrm{st}$ &  Eff& $\mcQ(u) - \mcQ(I u_{\tau h})$ & $\mcE_\mathrm{st}$ &  Eff\\    
      \hline 
      256   & 6.061e-05  & 2.583e-05&  0.426 & 4.184e-05  & 2.013e-05  &  0.481\\   
      2048   & 2.471e-05 &1.086e-05&  0.440  & 1.651e-05 &7.581e-06&  0.459\\   
     16384   &7.926e-06&  3.717e-06&  0.469 &5.295e-06 &2.534e-06&   0.479\\   
     131072   &2.233e-06&  1.082e-06&  0.485 &1.491e-06&  7.299e-07&  0.489 \\
      \hline 
\end{tabular}
\end{small}
\caption{{\small Effectivity of error estimate \cref{eq:ac_est_st} under space-time refinement with uniform initial time steps and nonuniform initial time steps}}\label{tab:ac_eff_table_n}
\end{table}

From all the results of \Cref{tab:ac_eff_table_s}--\Cref{tab:ac_eff_table_n}, we observe that the error and the estimate converge with the same order under various refinements. The effectivity indices are always between $0.4$ and $1$, and seem to converge to a constant. This indicates the asymptotic effectivity of the error estimate. 
\subsubsection*{Adaptivity test for the linear heat equation: Clockwise moving source}
Let $\Omega=(0,1) \times (0,2) $. We take the example from Asner, Tavener and Kay~\cite{asner2012adjoint} by choosing the right hand side and the initial and the boundary condition so that the exact solution is given by
$$
u(\mathbf{x}, t) = \exp(-100(x-0.5-0.25\sin(\pi t))^2 - 100(y-1-0.5\cos(\pi t))^2)
$$
which is an exponential peak moving clockwise inside the domain, see \Cref{fig:clockwise}. We set up the final time $T=0.5$ and $ \theta = \lambda = 0.8$, and start the adaptive procedure with a uniform mesh containing $4 \times 8$ elements and $10$ equally distributed time steps (i.e. $\tau = 0.05$). We aim to minimize the error at final time $T$, i.e., as before, $\mcQ(u) - \mcQ(I u_{\tau h}) = \Vert u(T) - I u_{\tau h}(T) \Vert_{L^2(\Omega)}^2 = \Vert u(T) - u_N^h \Vert_{L^2(\Omega)}^2$. For the sequential adaptive algorithm, we choose an initial coarse mesh with $4 \times 8$ elements and an initial $\tau = 0.01$. 
\par
\begin{figure}[thp]
\centering
\captionsetup[subfigure]{labelformat=empty}
\begin{tabular}{ll} 
\rotatebox{90}{{\scriptsize Exact} }&\hspace{-16pt}
  \subfloat{\includegraphics[scale=0.241]{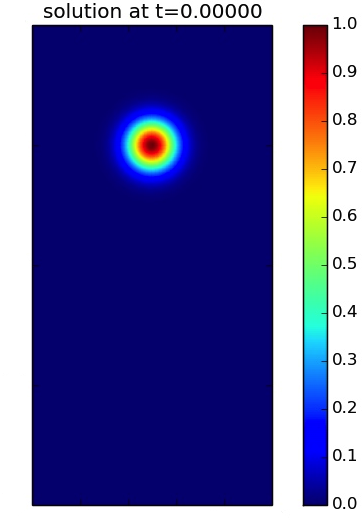}}
  \subfloat{\includegraphics[scale=0.241]{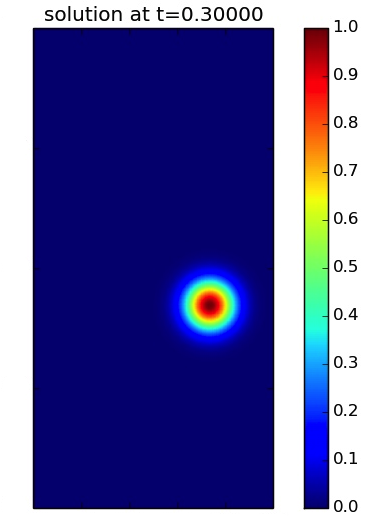}}
  \subfloat{\includegraphics[scale=0.241]{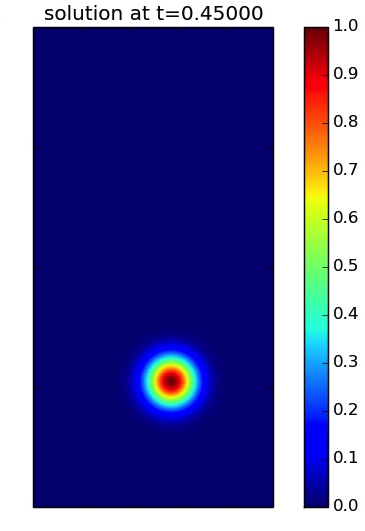}}
  \subfloat{\includegraphics[scale=0.241]{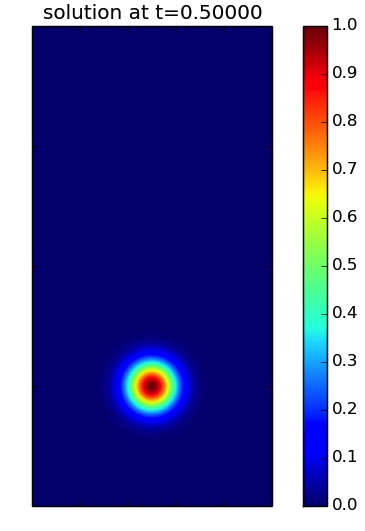}}\\[-10pt]
  \rotatebox{90}{{\scriptsize Sequential}}&\hspace{-16pt}
  \subfloat{\includegraphics[scale=0.241]{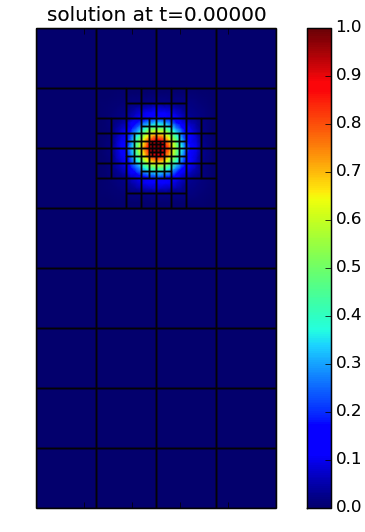}}
  \subfloat{\includegraphics[scale=0.241]{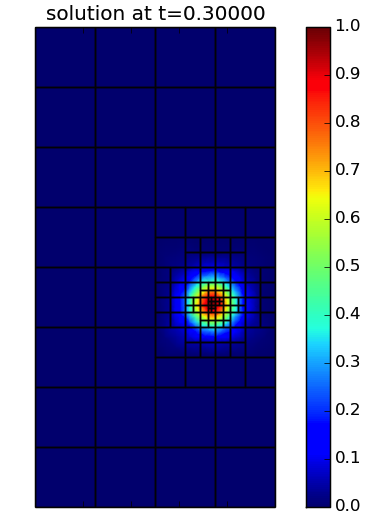}}
  \subfloat{\includegraphics[scale=0.241]{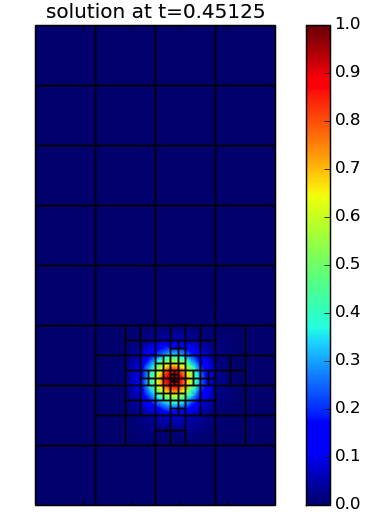}}
  \subfloat{\includegraphics[scale=0.241]{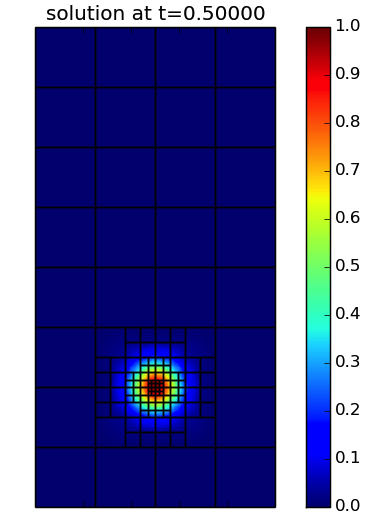}}\\[-10pt]
  \rotatebox{90}{{\scriptsize Duality-based}}&\hspace{-16pt}
  \subfloat[$t=0.0$]{\includegraphics[scale=0.241]{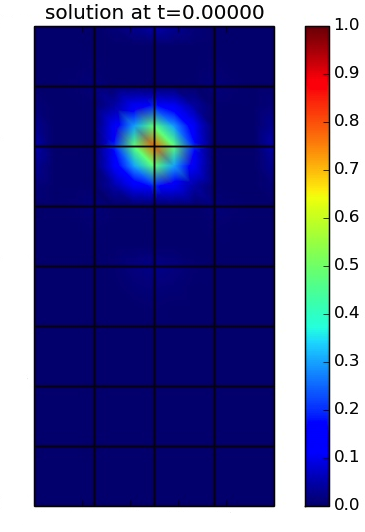}}
  \subfloat[$t=0.3$]{\includegraphics[scale=0.241]{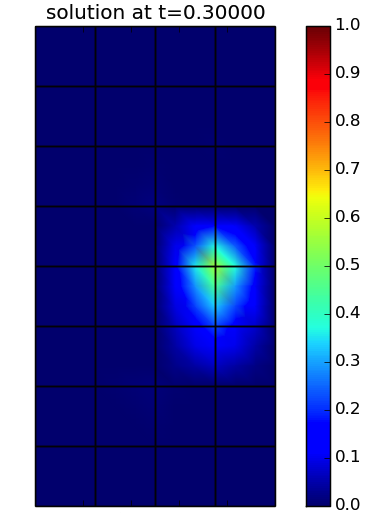}}
  \subfloat[$t=0.45$]{\includegraphics[scale=0.241]{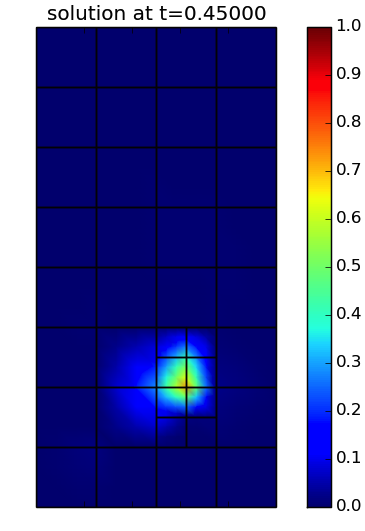}}
  \subfloat[$t=0.5$]{\includegraphics[scale=0.241]{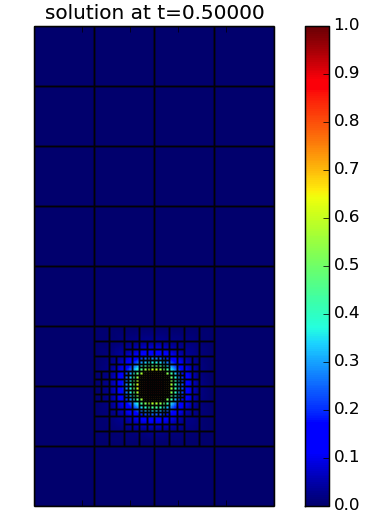}}\\[-8pt]
\end{tabular}
\caption{{\small Snapshot of the exact solution (first row), the sequential adaptive solution with the computational mesh (second row) and the duality-based adaptive solution with the computational mesh (third row) at several time points} }\label{fig:clockwise}
\end{figure}
In \Cref{fig:clockwise}, we show a comparison of the results obtained by the sequential adaptive algorithm and our duality-based adaptive algorithm. The first row corresponds to the exact solution. The second and third row are the snapshots of adaptively refined meshes with corresponding approximations by using the two adaptive algorithms. \Cref{fig:heat_steps} illustrates the various time steps over time for the two adaptive algorithms. A comparison of the convergence of the error is shown in \Cref{fig:heat_error} where `Total dof' refers to total number of degrees of freedom, $\sum_{k=0}^N M_k$. 
\begin{figure}[tbhp]
\begin{minipage}{.49\textwidth}
\centering
\includegraphics[scale=0.33]{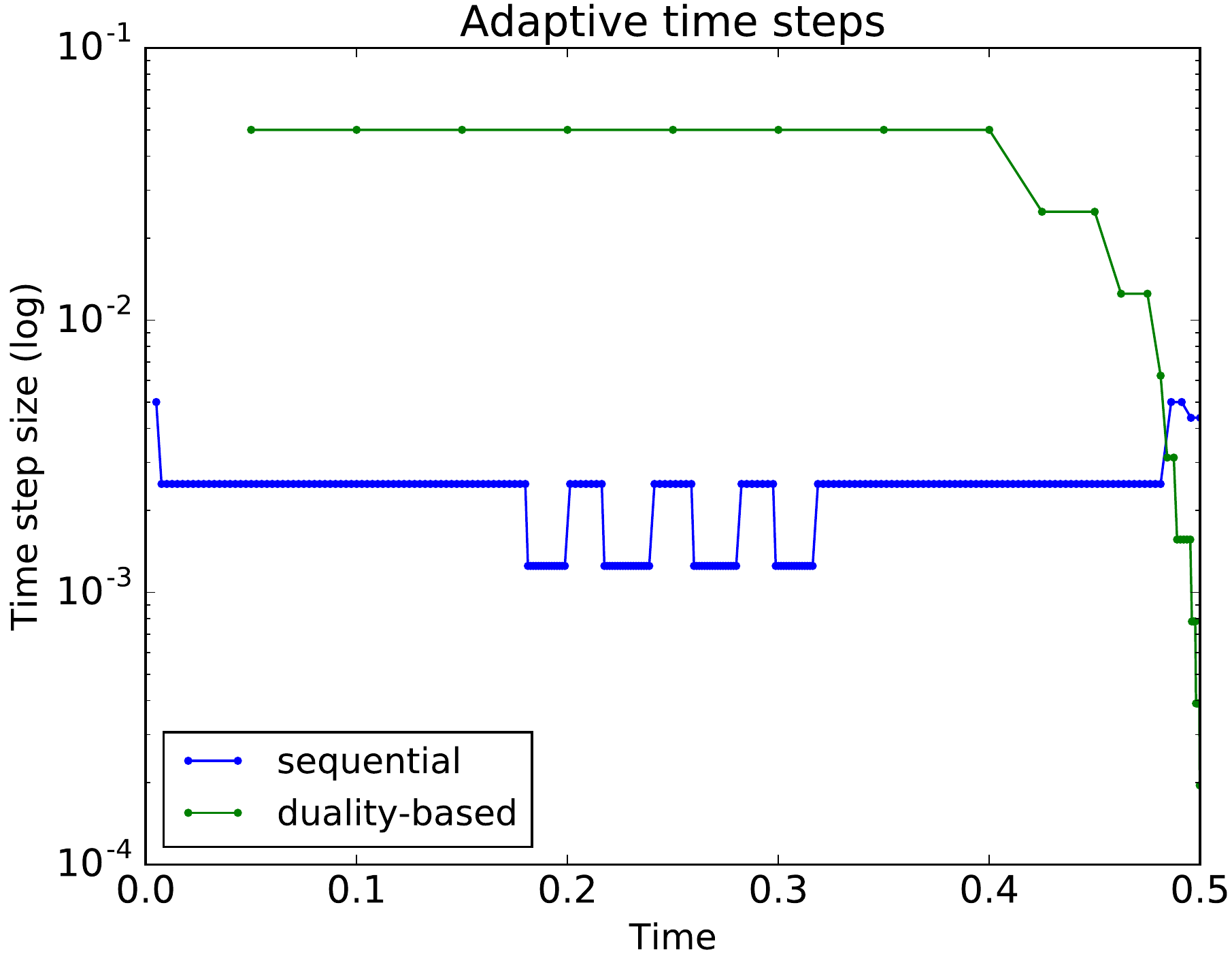} 
\caption{{\small Adaptive time step refinement for the sequential adaptive algorithm and the non-sequential duality-based adaptive algorithm}}\label{fig:heat_steps}
\end{minipage}
\hspace{2pt}
\begin{minipage}{.48\textwidth}
\centering
\includegraphics[scale=0.33]{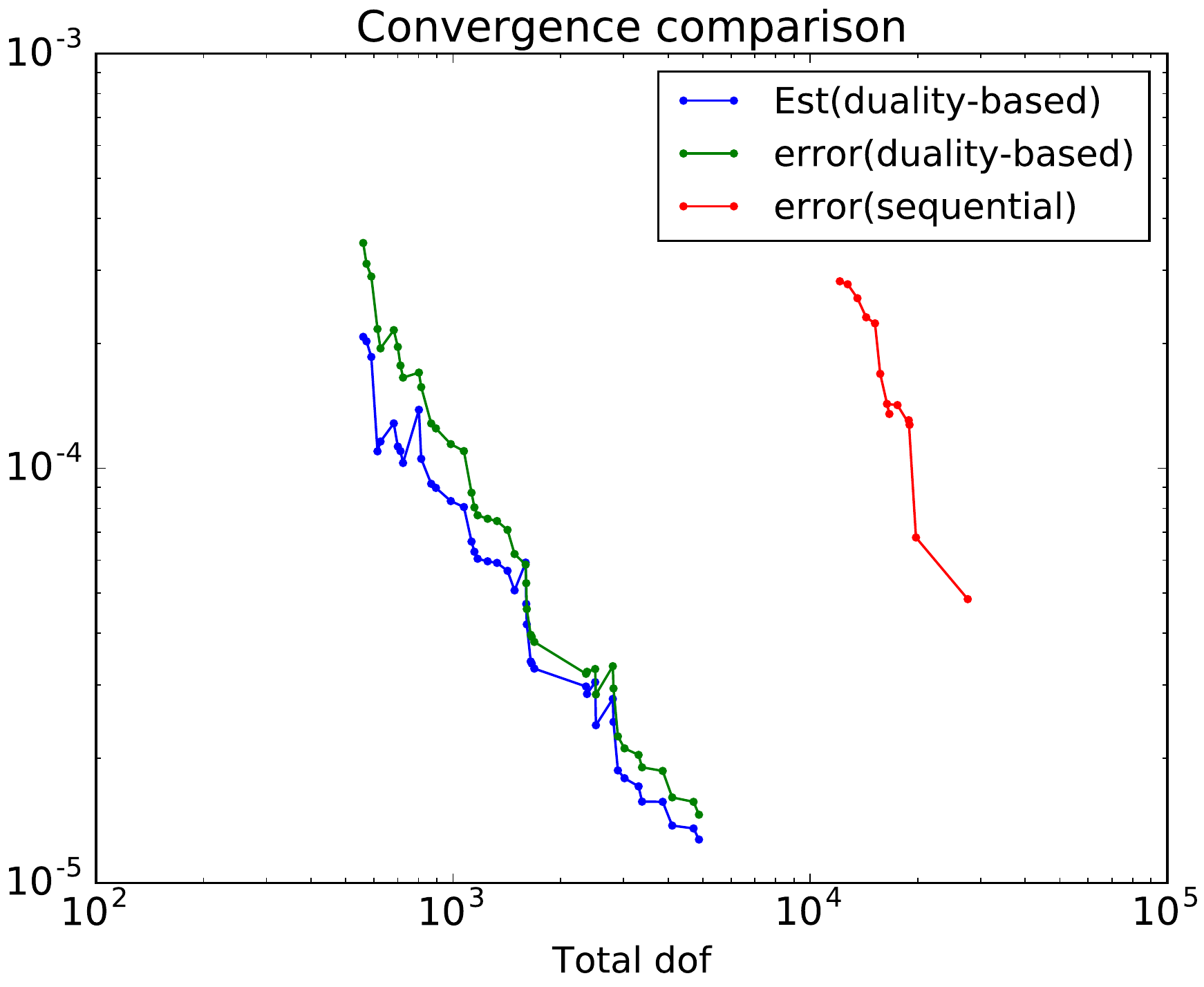}
\caption{{\small Convergence comparison between the sequential adaptive algorithm and the non-sequential duality-based adaptive algorithm}}\label{fig:heat_error}
\end{minipage}
\end{figure}
\par
From these results, it can be seen that the spatial and temporal refinements of the sequential adaptive algorithm focus on tracing the movement of the Gaussian shaped peak of the solution in the space-time domain. In contrast, since our duality-based adaptive algorithm targets the $L^2$ norm of the final error, the refinements of our algorithm are only concentrated at the final moment of the space-time domain. Away from this final moment, where residuals contribute much less to the final error, the original mesh and time step already provide sufficient resolution and need not be refined. This leads to fewer degrees of freedom and time steps for reaching the same accuracy of the final solution than for the sequential adaptive algorithm.
\par
In \Cref{fig:heat_error}, we also show the convergence of the error and the error estimate $\mcE_\mathrm{st}$ for the duality-based adaptive algorithm, which demonstrates the accuracy of the error estimate under adaptive refinement.  

\subsubsection*{Adaptivity test for the Allen--Cahn equation: Shrinking ring}~\\
The Allen--Cahn dynamics of this test case is a shrinking ring (with diffuse interfaces) in the middle of the domain $\Omega = (-1,1)^2$. The initial condition is set as 
\begin{equation}\label{ac_initial}
u(0) = -\tanh \left( \frac{\sqrt{x^2+y^2}-0.6}{\sqrt{2}\epsilon} \right) + \tanh \left( \frac{\sqrt{x^2+y^2}-0.15}{\sqrt{2}\epsilon} \right) -1
\end{equation}
where $\epsilon = 0.0625$, see \Cref{fig:AC_adp}. The inner circle has a small radius of $0.15$ which is expected to vanish much earlier than the outer circle. We are interested in the final error (i.e. $\mcQ(u)-\mcQ(I u_{\tau h})=\Vert u - I u_{\tau h} \Vert_{L^2}^2(T)$) when the inner circle has disappeared, i.e., the final time $T = 0.02$. To have a reference value for the error, we compute an approximation to \cref{eq:AC_full} on a uniform mesh with $512^2$ elements and a uniform time step size $\tau = 1e-5$. For the adaptive algorithm, we take a coarse initial time-step size $\tau_k = 5e-3$ and a coarse initial mesh with $16^2$ elements. The fractions in the adaptive algorithm are selected as $ \theta = \lambda = 0.8$. 
\begin{figure}[tbhp]
\centering
\captionsetup[subfigure]{labelformat=empty}
\begin{tabular}{ll}
  &\hspace{-14pt}
  \subfloat{\Plott{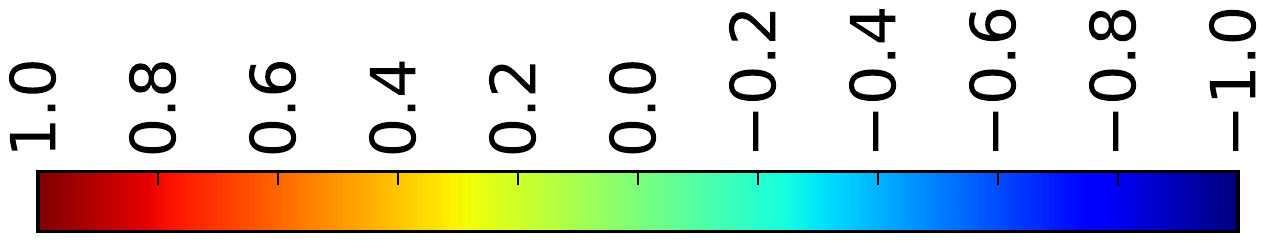}}\hspace{1pt}\!
    \subfloat{\Plott{color_bar.pdf}}\,\!
  \subfloat{\Plott{color_bar.pdf}}\,\!
  \subfloat{\Plott{color_bar.pdf}}\\[-12pt]
\rotatebox{90}{{\scriptsize Reference}}&\hspace{-12pt}
  \subfloat{\clippedFig{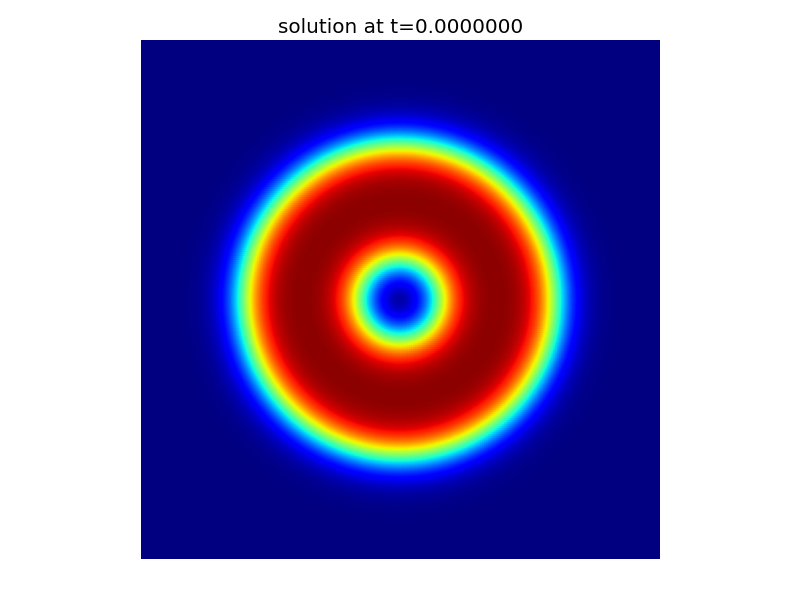}}\hspace{1pt}
    \subfloat{\clippedFig{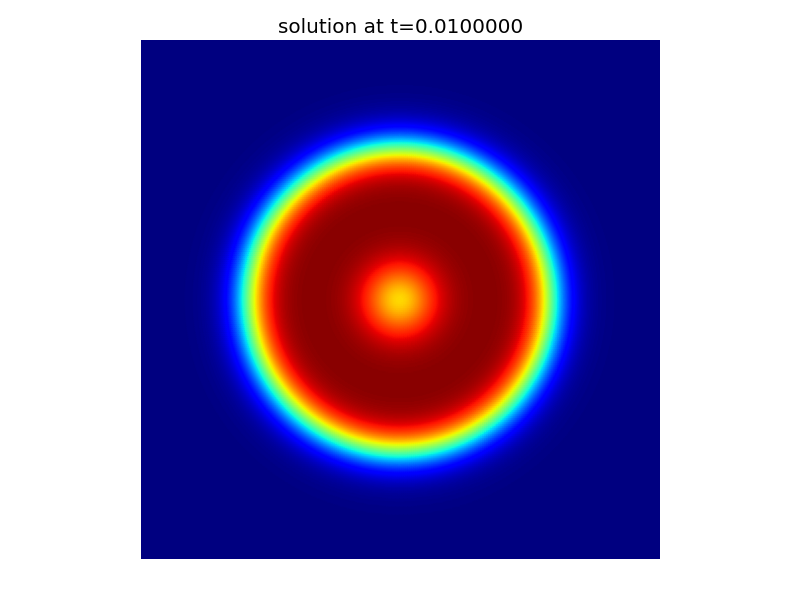}}\hspace{1pt}
  \subfloat{\clippedFig{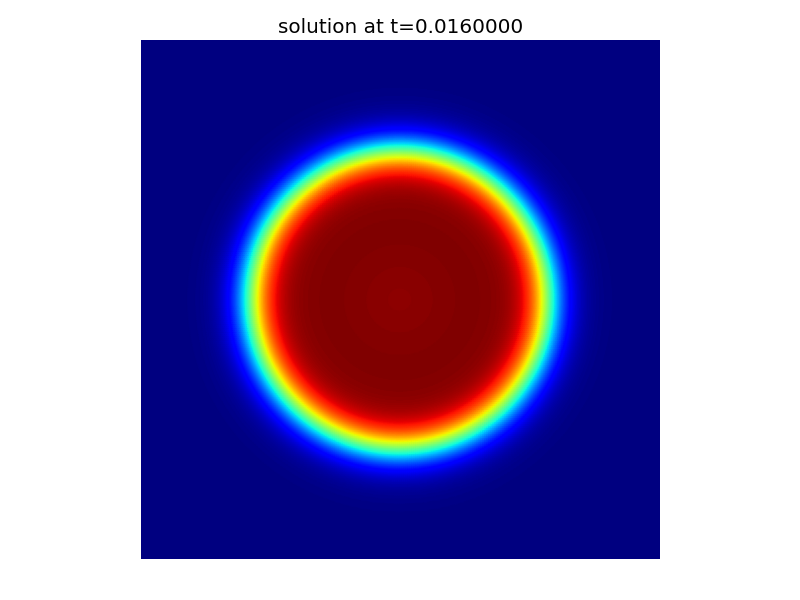}}\hspace{1pt}
  \subfloat{\clippedFig{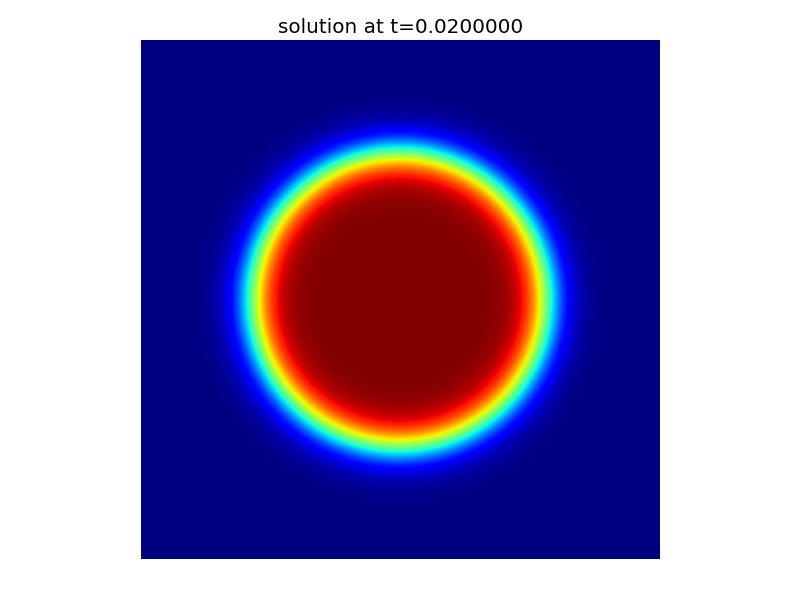}}\\[-8pt]
    &\hspace{-14pt}
  \subfloat{\Plott{color_bar.pdf}}\hspace{1pt}\!
    \subfloat{\Plott{color_bar.pdf}}\,\!
  \subfloat{\Plott{color_bar.pdf}}\,\!
  \subfloat{\Plott{color_bar.pdf}}\\[-12pt]
  \rotatebox{90}{{\scriptsize Primal}}&\hspace{-12pt}
  \subfloat{\clippedFig{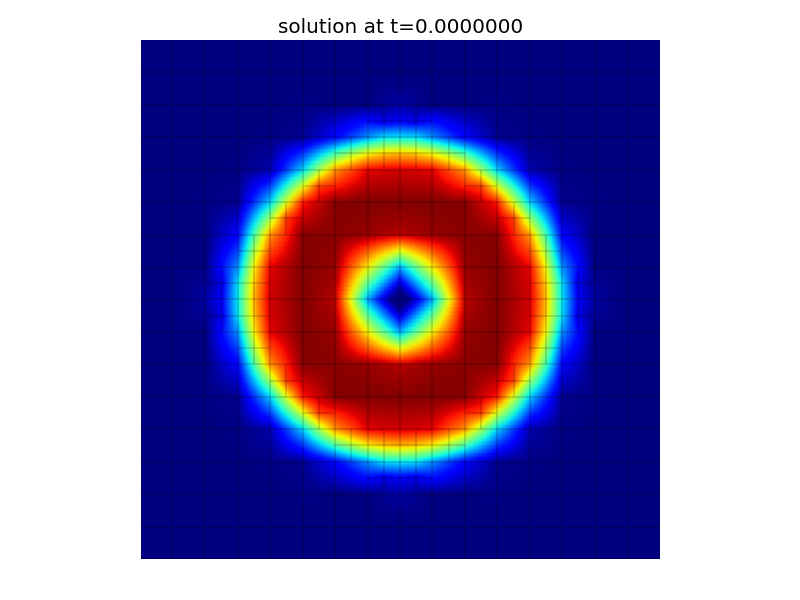}}\hspace{1pt}
    \subfloat{\clippedFig{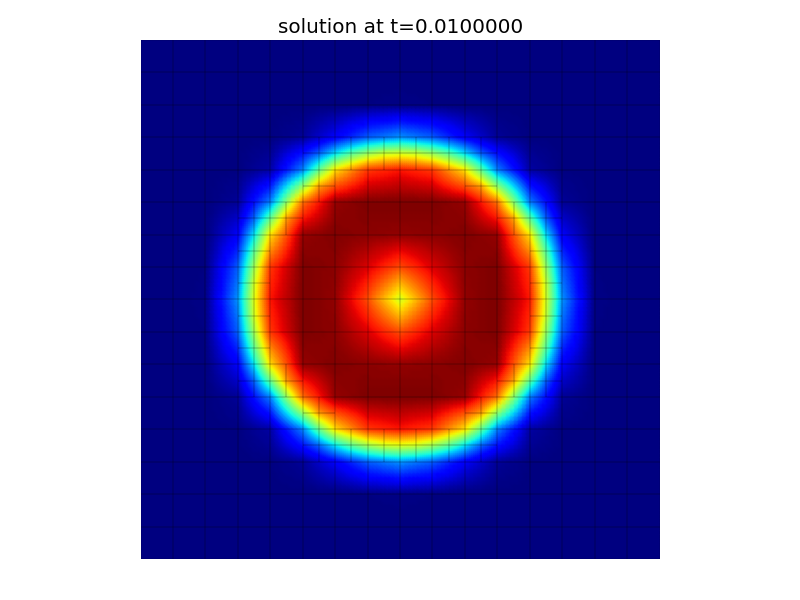}}\hspace{1pt}
  \subfloat{\clippedFig{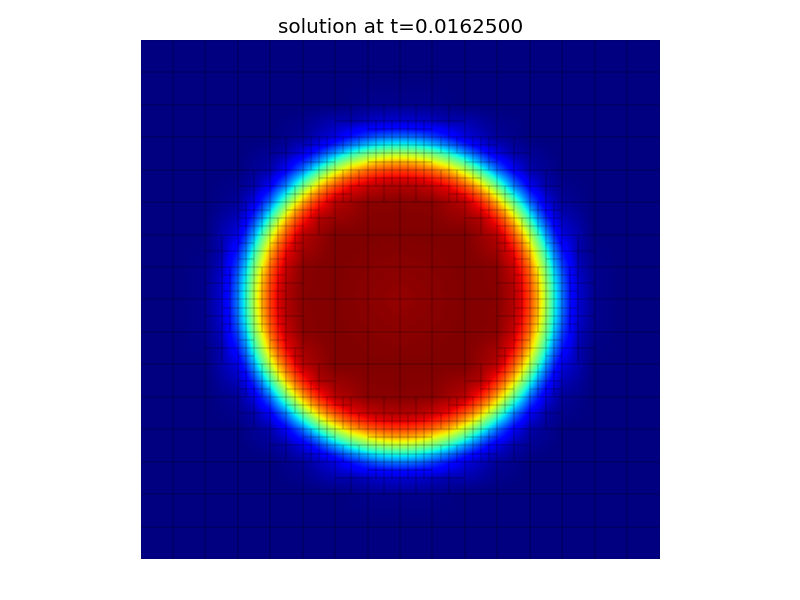}}\hspace{1pt}
  \subfloat{\clippedFig{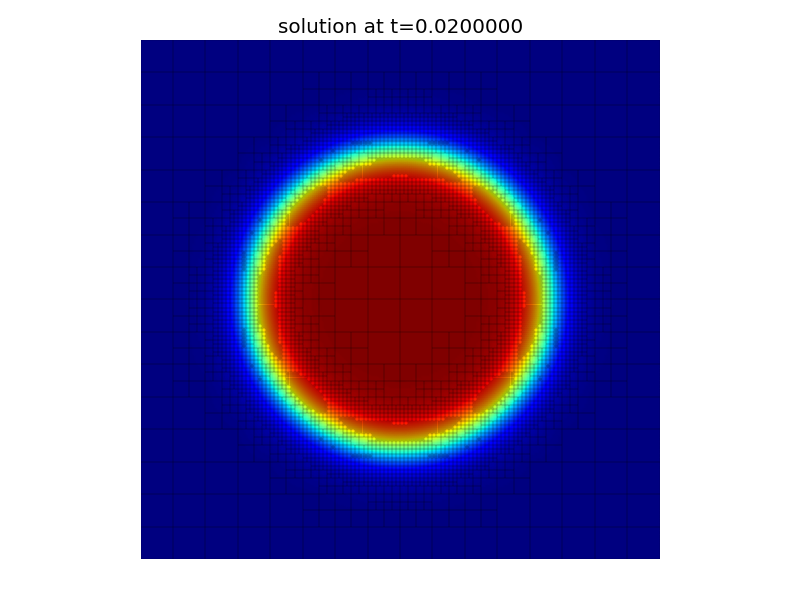}}\\[-8pt]
    &\hspace{-12pt}
  \subfloat{\Plott{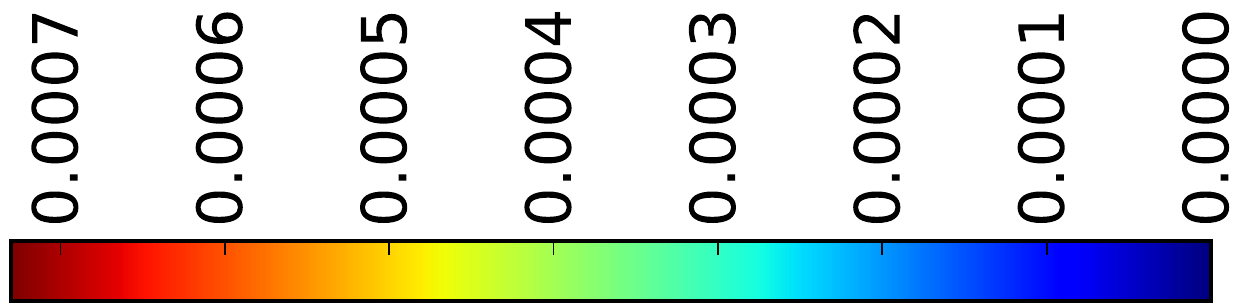}}\hspace{1pt}
    \subfloat{\Plott{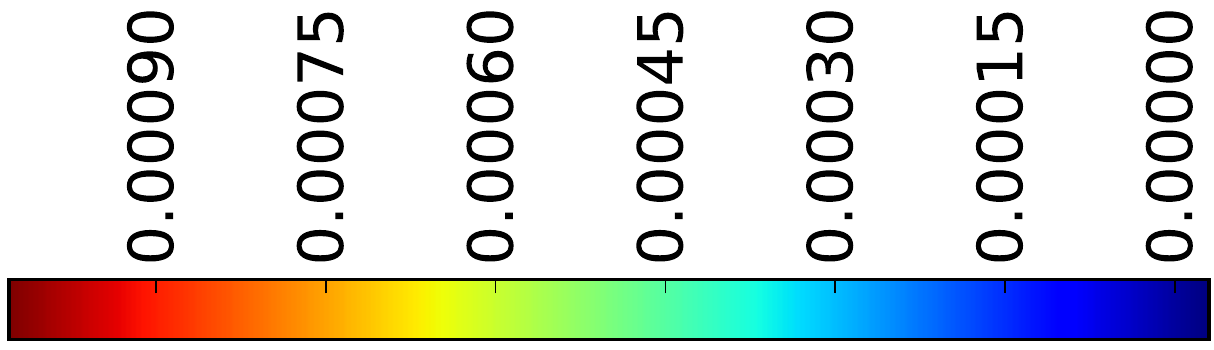}}\hspace{3pt}
  \subfloat{\Plott{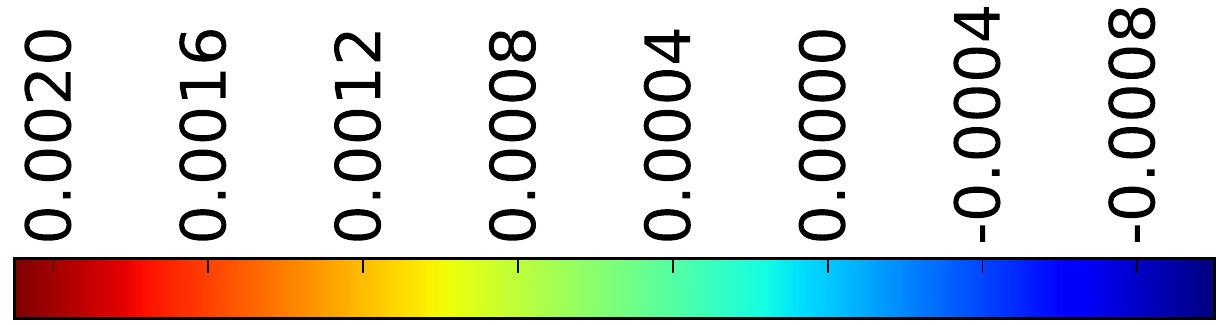}}\hspace{2pt}
  \subfloat{\Plott{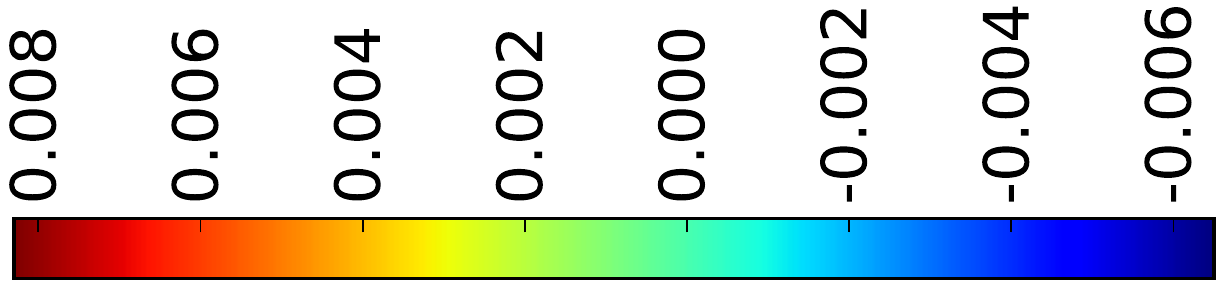}}\\[-12pt]
   \rotatebox{90}{{\scriptsize Dual}}&\hspace{-12pt}
  \subfloat{\clippedFig{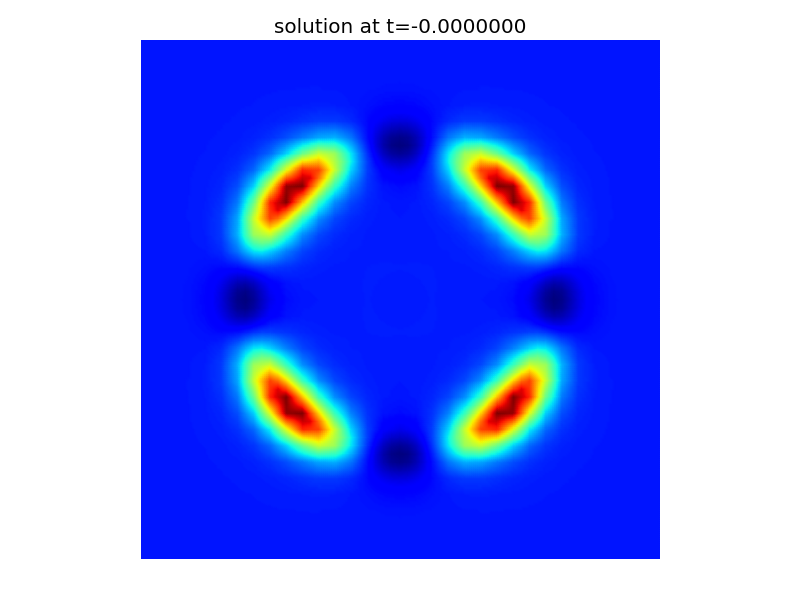}}\hspace{1pt}
    \subfloat{\clippedFig{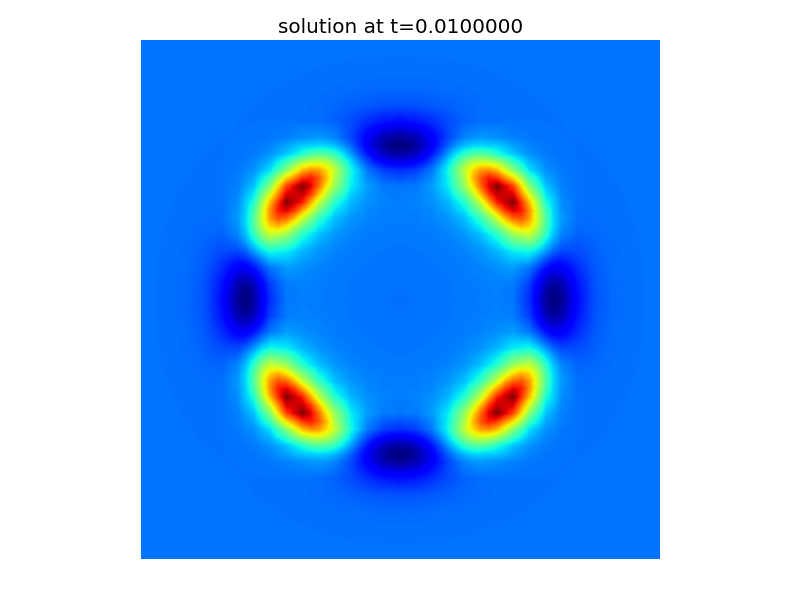}}\hspace{1pt}
  \subfloat{\clippedFig{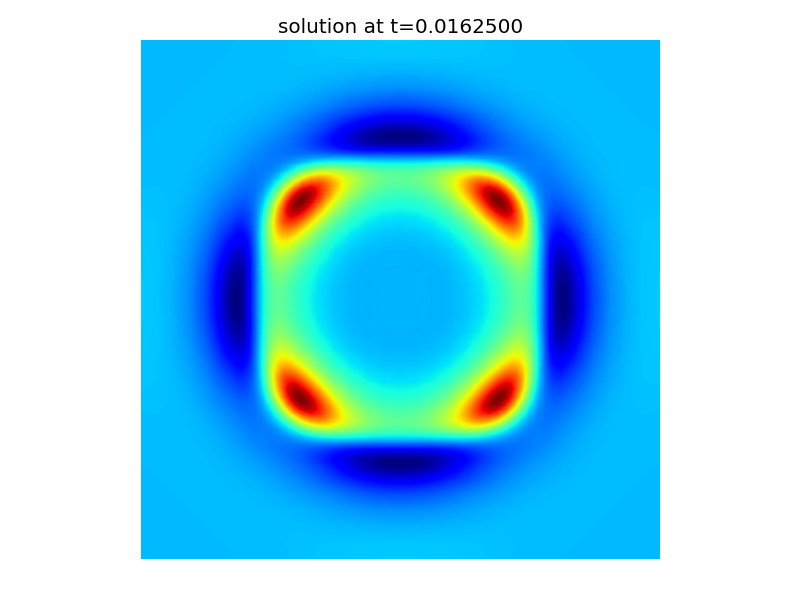}}\hspace{1pt}
  \subfloat{\clippedFig{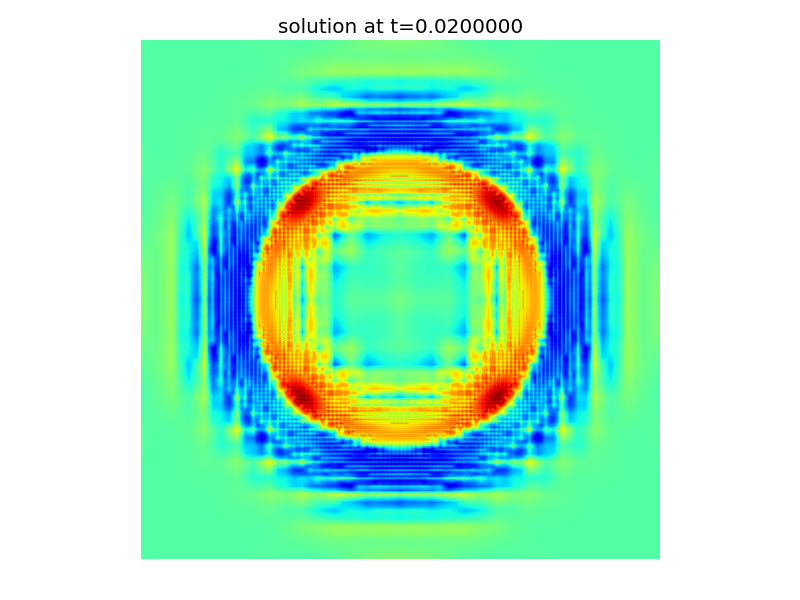}}\\[-5pt]
  \rotatebox{90}{{\scriptsize Spatial mesh}}&\hspace{-12pt}
  \subfloat{\clippedPlot{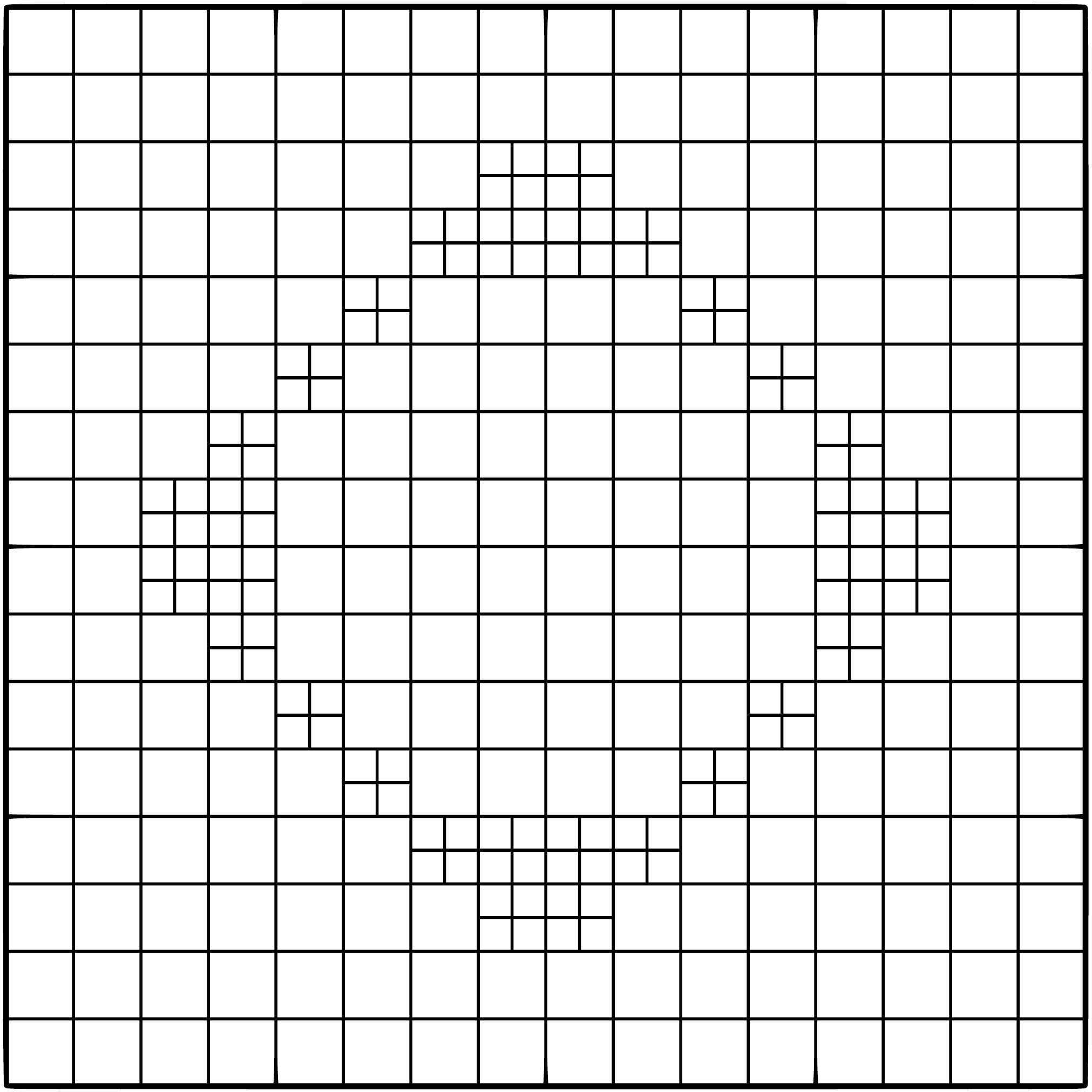}}\hspace{2pt}
    \subfloat{\clippedPlot{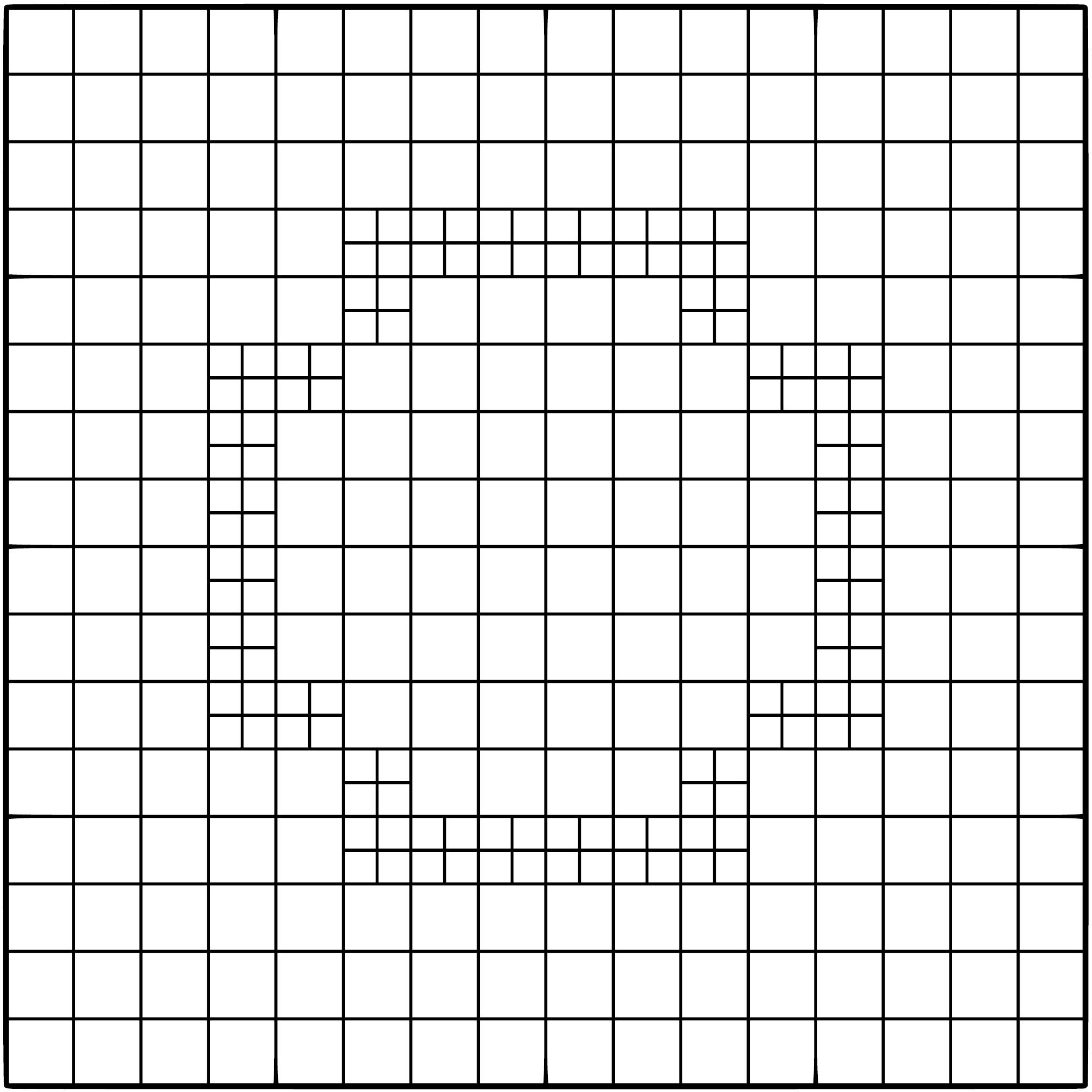}} \hspace{2pt}
  \clippedPlot{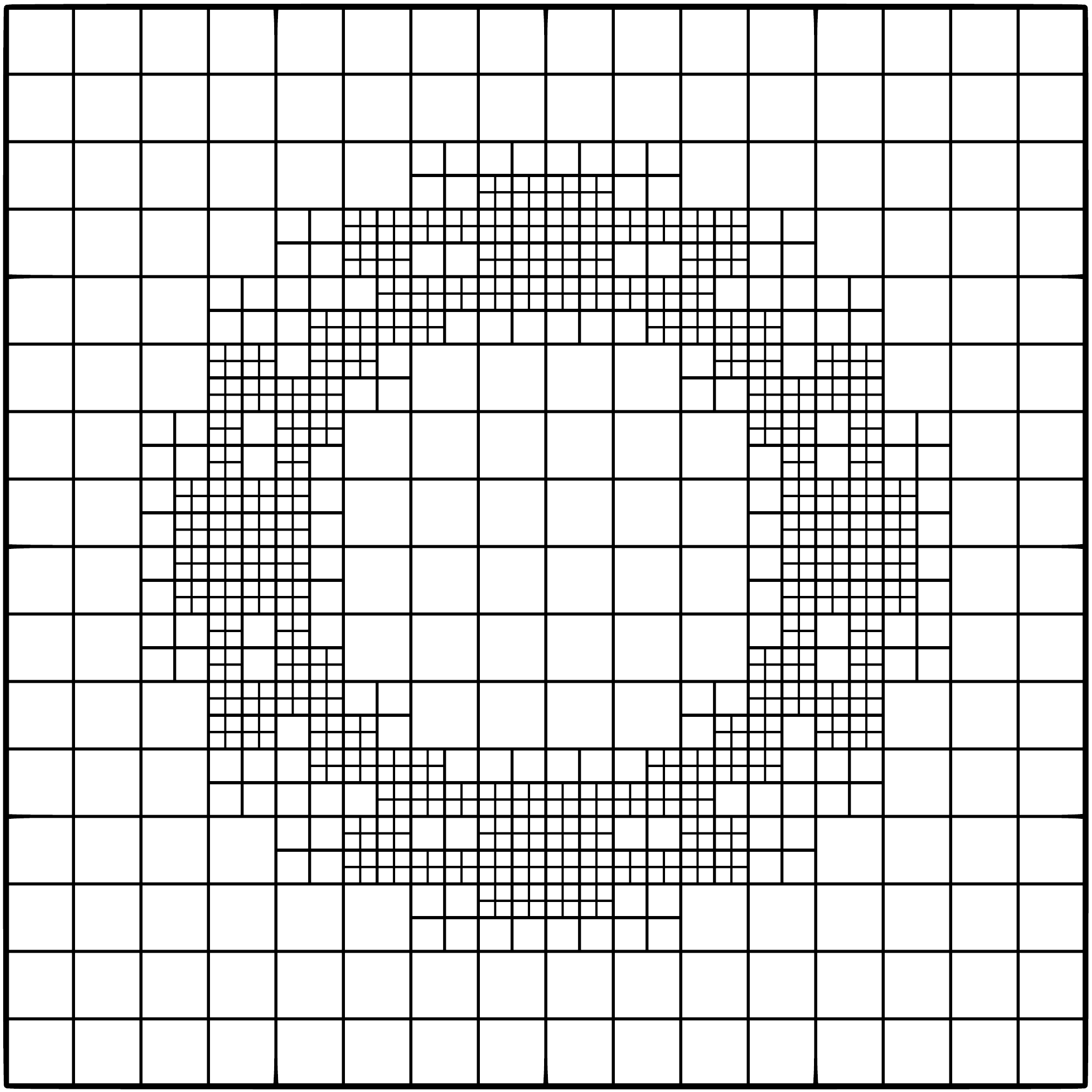}\hspace{2pt}
  \clippedPlot{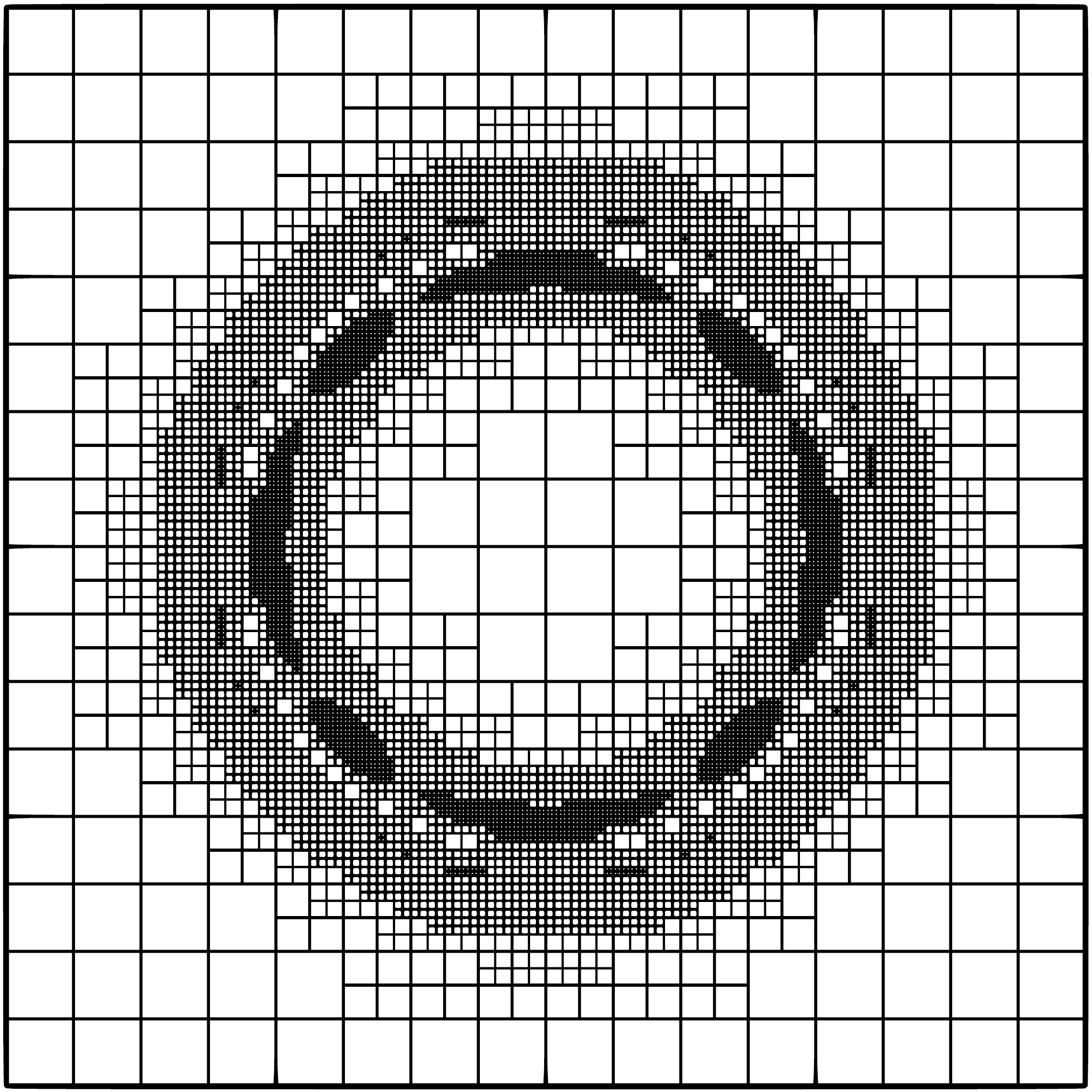}\\
  & {\scriptsize \hspace{24pt}$t=0.0$\hspace{60pt}$t=0.01$\hspace{58pt}$t=0.016$\hspace{55pt}$t=0.02$ }\\[-8pt]
\end{tabular}
\caption{{\small Shrinking ring: snapshot of the reference solution (first row), the adaptive primal solution (second row),  the adaptive dual solution (third row) and the computational mesh (fourth row) at several time points.}}\label{fig:AC_adp}
\end{figure}
\par
Snapshots of the results are presented in \Cref{fig:AC_adp}. The first row shows snapshots of the reference solution, while the second and fourth row show the primal approximation and the corresponding adaptive mesh obtained by the proposed adaptive algorithm. The computed dual solution $z$ is displayed in the third row. It can be observed that the dual solution grows as time progresses. This growth is localized at the interface of the outer ring. Let us mention that, for better visualization, in the plots of the dual solution the range of the color bars is adapted to each plot. \Cref{fig:ac_adp_step} displays the various time steps over time for the shrinking ring. As expected, smaller time steps are need at the vanishing moment of the inner ring and towards the final time $T$. Furthermore, the interfaces in the solution are well-resolved throughout the simulation time, with a significant increase of resolution towards the final time $T$. 
\par
The convergence of the error estimate and the error for the duality-based adaptive algorithm in comparison to uniform space-time refinement is shown in \Cref{fig:ac_adp_est} where `Total dof' refer to total number of degrees of freedom $\sum_{k=0}^N M_k$. Note that the error exhibits a plateau for the most refined approximations because the accuracy of the adaptively-refined approximations surpasses that of the reference approximation.
\begin{figure}[htbp]
\begin{minipage}{.49\textwidth}
\centering
\includegraphics[scale=0.33]{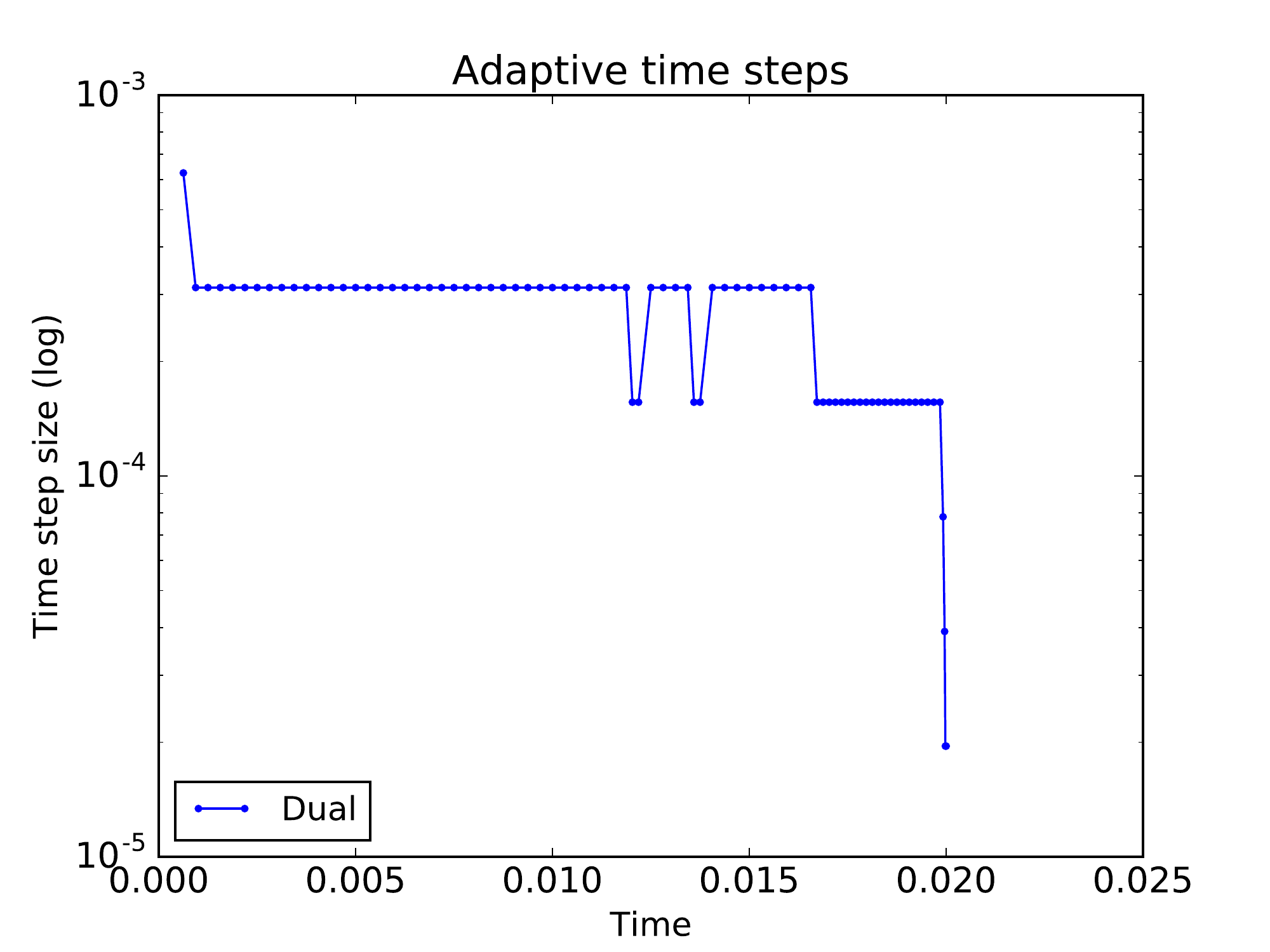} 
\caption{{\small Adaptive time step refinement for the duality-based adaptive algorithm}}\label{fig:ac_adp_step}
\end{minipage}
\hspace{2pt}
\begin{minipage}{.48\textwidth}
\centering
\includegraphics[scale=0.33]{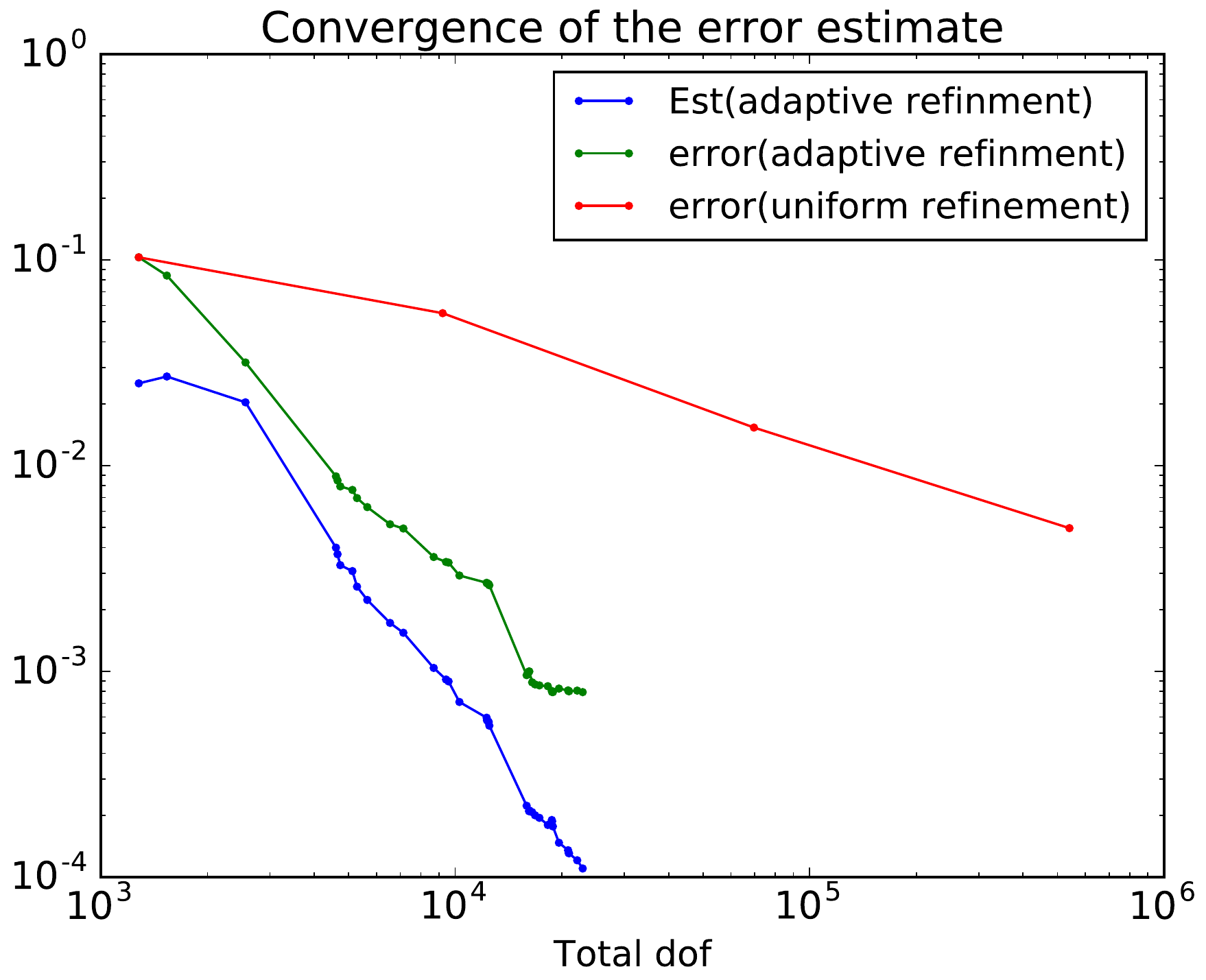}
\caption{{\small Convergence of the error estimate for the duality-based adaptive algorithm}}\label{fig:ac_adp_est}
\end{minipage}
\end{figure}

\section{Conclusion}
In this work we carried out a comprehensive study of duality-based a posteriori error estimates for semi-linear parabolic problems, with a special focus on discretizations using the finite element method in space combined with IMEX time stepping. 
We introduced a decomposition of the error estimates to identify the separate error contributions due to temporal and spatial approximation. The key idea is to adapt the residual decomposition by Verf{\"{u}}rth to our duality-based error representation and propose a specially-tailored time-discrete dual problem. The resultant error indicators quantify the spatial and temporal discretization errors and provide information to drive adaptive mesh refinement and adaptive time-step selection. \par
To illustrate the performance of the duality-based error estimates and the proposed adaptive algorithm, we presented numerical experiments for the heat equation and Allen--Cahn equation. We refer to~\cite[Section~6.3]{WuPhD2017} for the application to systems. The numerical results verified the accuracy and the effectivity of the error estimate in test problems. We also observed the overall good quality of the adaptive algorithm.\par
The proposed methodology can be further extended to other finite difference time-stepping schemes which do not fit in our considered abstract setting, e.g. higher-order multi-stage Runge-Kutta schemes. The key challenge in any such extension is the derivation of a specially-tailored time-discrete dual problem. Our analysis indicates that these dual problems can be derived systematically by means of ``backwards'' summation-by-parts on the time-discrete system.


\bibliographystyle{siamplain}
\bibliography{references}
\end{document}